\documentclass[11pt]{amsart}

\usepackage{amsmath}
\usepackage{amsfonts}
\usepackage{amssymb}
\usepackage{graphicx}
\usepackage[all,cmtip]{xy}
\usepackage{enumitem}
\usepackage{aliascnt} 
\theoremstyle{plain}
\newtheorem{theorem}{Theorem}[section]

\newaliascnt{corollary}{theorem}

\aliascntresetthe{corollary}

\newaliascnt{lemma}{theorem}
\newtheorem{lemma}[lemma]{Lemma}
\aliascntresetthe{lemma}
\newaliascnt{proposition}{theorem}

\aliascntresetthe{proposition}

\newaliascnt{hypotheses}{theorem}

\aliascntresetthe{hypotheses}

\newcommand{\C}{\mathbf{C}}
\newcommand{\R}{\mathbf{R}}
\newcommand{\Z}{\mathbf{Z}}
\newcommand{\Q}{\mathbf{Q}}
\newcommand{\GL}{\mathrm{GL}}
\newcommand{\SL}{\mathrm{SL}}
\newcommand{\SO}{\mathrm{SO}}
\newcommand{\bS}{\mathbf{S}}
\newcommand{\X}{\mathbf{X}}

\theoremstyle{definition}
\newaliascnt{definition}{theorem}
\newtheorem{definition}[definition]{Definition}
\aliascntresetthe{definition}

\newaliascnt{example}{theorem}

\aliascntresetthe{example}

\newaliascnt{remark}{theorem}

\aliascntresetthe{remark}

\newaliascnt{remarks}{theorem}

\aliascntresetthe{remarks}

\numberwithin{equation}{section}

\usepackage{calrsfs}

\usepackage{}
\usepackage[pdftex,breaklinks,colorlinks,
citecolor=blue,
linkcolor=blue,
urlcolor=blue]{hyperref}

\everymath{\displaystyle}

\begin{document}

\title[Elliptic units for complex cubic fields]{Elliptic units for complex cubic fields \\[1ex]
\small {\it (On Eisenstein's Jugendtraum)}}

\author{Nicolas Bergeron}
\address{DMA UMR 8553 ENS / PSL et Sorbonne Université, F-75005, Paris, France}
\email{nicolas.bergeron@ens.fr}
\urladdr{https://sites.google.com/view/nicolasbergeron/accueil}

\author{Pierre Charollois}
\address{Sorbonne Universit\'e, IMJ--PRG, CNRS, Univ Paris Diderot, F-75005, Paris, France}
\email{pierre.charollois@imj-prg.fr}
\urladdr{https://webusers.imj-prg.fr/~pierre.charollois}

\author{Luis E. Garc\'ia}
\address{Department of Mathematics, University College London, Gower Street, London WC1E 6BT, United Kingdom}
\email{l.e.garcia@ucl.ac.uk}
\urladdr{https://www.ucl.ac.uk/~ucahljg/index.htm}

\begin{abstract}

We propose a conjecture extending the classical construction of elliptic units to complex cubic number fields $K$. The conjecture concerns special values of the elliptic gamma function, a holomorphic function of three complex variables arising in mathematical physics whose transformation properties under $\mathrm{SL}_3(\Z)$ were studied by Felder and Varchenko in the early 2000s. Using this function we construct complex numbers that we conjecture to be units in narrow ray class fields of $K$. We also propose a reciprocity law for the action of the Galois group on these units in the style of Shimura. 

To support our conjecture we offer numerical evidence and also prove a new type of Kronecker limit formula relating the logarithm of the modulus of these complex numbers to the derivatives at $s=0$ of partial zeta functions of $K$. Our constructions unveil the role played by the elliptic gamma function in Hilbert's twelfth problem for complex cubic fields.  

\end{abstract}
\maketitle

\begin{flushright} {\it To Robert Sczech with admiration}
\end{flushright}

\setcounter{tocdepth}{1}

\tableofcontents

\section{Introduction}

Elliptic units are certain special units in abelian extensions of imaginary quadratic number fields. One key object of this theory is the theta function
$$
\theta_0 (z , \tau  ) =\prod_{n \geq 0} (1-e^{2\pi i (n\tau +z)}) (1-e^{2\pi i ((n+1) \tau - z)}).
$$
The infinite product converges absolutely when $\mathrm{Im}(\tau)>0$ and defines a holomorphic function on the product of the complex plane by the complex upper half-plane $\mathcal{H}$. Up to a factor $ie^{\pi i (z-\tau/6)} \eta (\tau)$, where $\eta$ is the Dedekind eta function, it is the first (odd) Jacobi theta function. The function $\theta_0$ enjoys the periodicity properties 
\begin{equation*}
\begin{split}
\theta_0(z+1,\tau) &= \theta_0(z,\tau) \\
\theta_0(z+\tau,\tau) &= -e^{-2\pi i z}\theta_0(z,\tau)
\end{split}
\end{equation*}
in the variable $z$, as well as well-known transformation properties in the variable $\tau$ under the usual action of the group $\mathrm{SL}_2(\Z)$ on $\C \times \mathcal{H}$.

Taking quotients of values of $\theta_0$ at special points $(v,\tau)$, where $v$ is a rational number and $\tau$ is a quadratic irrationality in $\mathcal{H}$, one obtains elliptic units. For example, let $v=1/3 \in \Q$ and $\tau=(2+i)/5 \in \Q(i)$. Then
$$\theta_0 (v ,\tau)^5 \cdot \theta_0 (5v , 5\tau)^{-1}$$
belongs to an abelian extension of $\Q (i)$ ramified only above $3$. This is a general phenomenon: similar expressions for other $v$ and $\tau$ allow for the analytic construction of units in abelian extensions of imaginary quadratic fields, encode special values of zeta functions through the Kronecker limit formula, and led Stark to his famous conjectures on the existence of special units related to Hecke L-functions. We refer to \cite{Robert,CoatesWiles,deShalit,Rubin} for beautiful expositions of this subject. 

Our article is motivated by the desire to transpose these classical results to the setting of complex cubic fields. We propose a conjecture that is very similar to the theory of elliptic units. The cornerstone analytic function becomes the elliptic gamma function. For complex numbers $z$, $\tau$,  $\sigma$ with $\mathrm{Im}(\tau)>0$ and $\mathrm{Im}(\sigma)>0$, it is defined by the absolutely convergent infinite product
\begin{equation} \label{eq:ell_gamma_intro}
\Gamma (z , \tau , \sigma ) = \prod_{j,k=0}^\infty \frac{1-e^{2\pi i ((j+1)\tau + (k+1) \sigma-z)}}{1-e^{2\pi i (j \tau + k \sigma +z)}}.
\end{equation}
Its domain is then extended to $(z,\tau,\sigma) \in \C \times (\C-\R)^2$ by the rules
$$
\Gamma (z , \tau , \sigma ) =  \frac{1}{\Gamma (z -\tau , -\tau , \sigma )} = \frac{1}{\Gamma (z -\sigma , \tau , -\sigma )}.
$$

Introduced by Ruijsenaars \cite{Ruijsenaard}, this function appears implicitly in Baxter's work on the eight-vertex model in statistical physics. It is a solution of an elliptic version of Euler's functional equation $\Gamma (z+1) = z \Gamma (z)$, in which the rational function $z$ is replaced by the above theta function: we have
\begin{equation*}
\begin{split}
\Gamma (z+1 , \tau , \sigma ) &= \Gamma (z, \tau , \sigma) \\
\Gamma(z+\tau,\tau,\sigma) &= \theta_0(z,\sigma) \Gamma(z,\tau,\sigma),
\end{split}
\end{equation*}
and a similar relation between $\Gamma(z+\sigma,\tau,\sigma)$ and $\Gamma(z,\tau,\sigma)$.
Now three periods $1$, $\tau$, $\sigma$ are involved, and one might expect the elliptic gamma function to have
$\SL_3 (\Z)$-modular properties. This was shown to be the case by Felder and Varchenko \cite{FelderVarchenko}, who studied this function in depth and established identities such as
$$
\Gamma(z/\tau, -1/\tau, \sigma/\tau) = e^{i \pi P(z;\tau,\sigma)} \Gamma((z-\tau)/\sigma,-\tau/\sigma,-1/\sigma) \Gamma(z,\tau,\sigma)
$$
for an explicit polynomial $P \in \Q(\tau,\sigma)[z]$. The resulting transformation properties of $\Gamma(z,\tau,\sigma)$ are higher analogues of the $\SL_2 (\Z)$-modular properties of the theta function $\theta_0(z,\tau)$ \cite{FelderVarchenko,Felder}.

Let $K \subset \C$ be a complex cubic number field. Along the lines of Hilbert's 12th problem we propose a way to produce complex numbers conjecturally lying in specific abelian extensions of $K$. These numbers are obtained by taking products and quotients of the elliptic gamma function evaluated at special points $(v , \tau , \sigma)$ belonging to $K^3$. Let us give a first example before stating the general conjecture.

\bigskip
\noindent {\it Example.} Consider the complex cubic field $K=\Q (\beta)$ with $\beta = \sqrt[3]{7} \cdot e^{-\frac{2i\pi}3 }$, so that $\beta^3 =7$. Define the point
$$
[\tau:\sigma:1] \in \C\mathrm{P}^2-\R\mathrm{P}^2, \quad \tau =\frac{2\beta+\beta^2}{15} \in \overline{\mathcal H} , \quad \sigma =-\frac{2+\beta}{15}\in \mathcal H.$$
The complex number
\begin{equation} \label{Intro:sol} 
\Gamma\left(1/3, \tau , \sigma \right)^5 \cdot \Gamma\left(5/3, 5\tau , 5\sigma \right)^{-1}\approx -4.024029545\ldots -i\cdot 41.85595177\ldots
\end{equation}
coincides (to 1000 digits at least) with a complex root of the polynomial
\begin{multline} \label{Intro:pol}
Q = x^6 + (6\beta^2-14 \beta +2) x^5 + (4 \beta^2 -6 \beta +2) x^4  \\ + (106 \beta^2 -152 \beta - 103) x^3 +(4 \beta^2 -6 \beta +2) x^2 + (6\beta^2-14 \beta +2) x
+ 1.
\end{multline}
This polynomial is irreducible over $K$ and its splitting field $H$ is a cyclic totally complex abelian extension of $K$ of degree $6$ ramified only at $3$. The roots of $Q$ are units in $H$; the polynomial $Q$ being palindromic, the inverse of \eqref{Intro:sol} is also a root of $Q$. The other four roots are (numerically with similar precision) obtained again  using the same formula \eqref{Intro:sol} now with
$$(\tau , \sigma) = \left(\frac{\beta^2+2\beta+75}{345} , - \frac{\beta+32}{345} \right) \quad \mbox{and} \quad  \left(\frac{\beta^2+2\beta+15}{150} , - \frac{\beta-43}{150} \right).$$

\vspace{0.7cm}
Returning to the general setting, let us consider an arbitary complex cubic field $K \subset \C$ and fix a proper ideal $\mathfrak{f}$ of its ring of integers $\mathcal{O}_K$. Class field theory shows that there exists a number field $K(\mathfrak{f}) \subset \C$ containing $K$ that is maximal among abelian extensions of $K$ ramified only at $\mathfrak{f}$. The field $K(\mathfrak{f})$ is called the narrow ray class field of conductor $\mathfrak{f}$.  In Section \ref{section:values_of_elliptic_gamma}, assuming that $K(\mathfrak{f})$ is totally complex, we introduce complex numbers $\mathbf{\Gamma}_{\mathfrak{a},h}(\mathfrak{fb}^{-1})$, depending on $\mathfrak{f}$ as well as on some auxiliary data, obtained as a product of values of the elliptic gamma function \eqref{eq:ell_gamma_intro} evaluated at elements of $K$. 

Our main conjecture is as follows (see Conjecture \ref{S:conj}, we highlight here only two of the three claims):

\noindent {\bf Conjecture.} \begin{enumerate}
    \item The complex numbers $\mathbf{\Gamma}_{\mathfrak{a},h}(\mathfrak{fb}^{-1})$ belong to $K(\mathfrak{f})$. More precisely, they are {\it units} in this number field, that is $\mathbf{\Gamma}_{\mathfrak{a},h}(\mathfrak{fb}^{-1}) \in \mathcal{O}_{K(\mathfrak{f})}^\times$.
    \item The action of the Galois group $\mathrm{Gal}(K(\mathfrak{f})/K)$ on the units $\mathbf{\Gamma}_{\mathfrak{a},h}(\mathfrak{fb}^{-1})$ can be described explicitly: if $\sigma_\mathfrak{b} \in \mathrm{Gal}(K(\mathfrak{f})/K)$ is the element assigned to an ideal $\mathfrak{b}$ relatively prime to $\mathfrak{f}$ by the Artin map, then we have the {\it reciprocity law}
    $$
    \sigma_\mathfrak{b}(\mathbf{\Gamma}_{\mathfrak{a},h}(\mathfrak{f})) = \mathbf{\Gamma}_{\mathfrak{a},h}(\mathfrak{fb}^{-1}).
    $$
\end{enumerate}

These properties are very similar to the expected properties of units whose existence is predicted by celebrated conjectures of Stark, known as {\it Stark units}. In Section \ref{section:relation_with_Starks_conjecture} we describe the precise relation between the two units. We show that our conjecture provides a formula for Stark units themselves rather than just their absolute values. This answers (conjecturally) a question plainly raised by Samit Dasgupta in \cite{Dasgupta}, and unexpectedly relates the elliptic gamma function to Hilbert's 12th problem asking for transcendental functions that provide explicit constructions of class fields.

Throughout the paper we  consider several examples and provide numerical evidence for the above properties. We pay special attention to an example studied by Samit Dasgupta \cite{Dasgupta} and we relate our conjectural units to the alleged Stark unit he found. Building on a subtle study of Shintani zeta functions  associated to cones, Ren and Sczech \cite{RenSczech} have proposed another conjectural formula for the Stark unit. We numerically check on their examples that our formulas give the same complex numbers.

Going beyond numerical evidence, our main result (Theorem \ref{T:KLintro}) relates our elliptic gamma values to the first derivative at $s=0$ of partial zeta functions of $K$. For an ideal $\mathfrak{b}$ coprime to $\mathfrak{f}$, define the partial zeta function
$$
\zeta_\mathfrak{f}(\mathfrak{b},s) = \sum_{\mathfrak{c} \in [\mathfrak{b}]} \frac{1}{\mathrm{N}(\mathfrak{c})^{s}}, \quad \mathrm{Re}(s) > 1.
$$
Here the sum runs over integral ideals $\mathfrak{c}$ defining the same class as $\mathfrak{b}$ in the narrow ray class group of conductor $\mathfrak{f}$ and $\mathrm{N}(\mathfrak{c})$ denotes the absolute norm of $\mathfrak{c}$. These functions admit meromorphic continuation to $s \in \C$ that vanishes at $s=0$. For ideals $\mathfrak{a}$, $\mathfrak{b}$ coprime with $\mathfrak{f}$ we set
$$
\zeta'_{\mathfrak{f},\mathfrak{a}}(\mathfrak{b},0) = \frac{d}{ds} \left( \zeta_{\mathfrak{f}}(\mathfrak{ab},s) - \mathrm{N}(\mathfrak{a}) \zeta_{\mathfrak{f}}(\mathfrak{b},s) \right)|_{s=0}.
$$

\bigskip

\noindent {\bf Theorem.}
{\it Assume that $\mathfrak{a}$, $\mathfrak{b}$ and $\mathfrak{f}$ satisfy the conditions in Section  \ref{subsection:complex_cubic_fields_and_elliptic_gamma_function}. Then}
$$
\zeta_{\mathfrak{f},\mathfrak{a}}'(\mathfrak{b},0) = \log |\mathbf{\Gamma}_{\mathfrak{a},h}(\mathfrak{fb}^{-1}) |^2.
$$

\bigskip

This theorem is an analogue, in the setting of complex cubic fields, of the classical Kronecker limit formula. It will be used in conjunction with our main conjecture to relate our units to Stark units. We have also verified this formula numerically in the examples considered in the paper. Coming back to the example above, the field $H$ equals $K(\mathfrak{f})$ for $\mathfrak{f}$ the unique ideal of norm $3$ in $K$, and for appropriate choices of $\mathfrak{a}$ and $\mathfrak{b}$ our Theorem shows that 
\begin{equation*}
\zeta'_{\mathfrak{f},\mathfrak{a}}(\mathfrak{b},0) = \log \left|\frac{\Gamma\left(1/3, \tau , \sigma \right)^5}{\Gamma\left(5/3, 5\tau , 5\sigma \right)} \right|^2.
\end{equation*}

For the proof of Theorem \ref{T:KLintro} we consider a fiber bundle over the locally symmetric space of a congruence subgroup of $\mathrm{SL}_3(\Z)$, with fiber the period domain $\C\mathrm{P}^2- \R \mathrm{P}^2$. We then introduce Eisenstein series on this space and relate two periods for them. The first, along a closed geodesic, equals the derivative at $0$ of a partial zeta function. We prove that the second period, along a new kind of modular symbol, equals the logarithm of the absolute value of products and quotients of the elliptic gamma function. Our Eisenstein series arise naturally from the equivariant point of view developed in \cite{BCG_CRM}; the period comparison is similar to the main argument in \cite{ColmezNous}.

Several authors have worked on generalising Kronecker's limit formulas and the related question of finding the correct analog of $|\eta(\tau)|$, see e.g. \cite{Efrat2,Efrat1, BekkiHecke}. The resulting formulas take a complicated form involving integration and are not simply obtained by taking $\log |\cdot|$ of a function with modular properties. Note also that the elliptic gamma function is not just an automorphic form but part of an ``automorphic $1$-cocycle'' for $\mathrm{SL}_3(\Z)$~\cite{FelderVarchenko}.

The geometric picture underlying our main construction is analogous to that of the recent work of Darmon and Vonk \cite{DarmonVonk} on explicit class field theory for real quadratic fields. In general other related works, like \cite{DD, CharolloisDarmon, DarmonL, DarmonPozziVonk, DK}, are either non-archimedean or restricted to other types of fields.

\medskip

\noindent {\bf Acknowledgment.} We would like to thank Hohto Bekki for his careful reading and comments on a preliminary draft.

\section{Values of the elliptic gamma function -- a conjecture}\label{section:values_of_elliptic_gamma}

In this section we state our main conjecture. After reviewing the definition and main properties of the elliptic gamma function in \textsection \ref{subsection:elliptic_gamma_function}, we proceed to define a variant of this function depending on some arithmetic data, involving a choice of complex cubic field and some ideals, in \textsection \ref{subsection:complex_cubic_fields_and_elliptic_gamma_function}. We then state our conjecture regarding the algebraicity and integrality of certain values of these functions, as well as a reciprocity law for the Galois action, in \textsection \ref{S:conj}. A first numerical example is discussed in detail in \textsection \ref{subsection:first_numerical_example}.

\subsection{The elliptic gamma function} \label{subsection:elliptic_gamma_function}

Let $L$ be a free abelian group of rank three and write $\Lambda$ for its dual $\mathrm{Hom}_{\Z} ( L , \Z )$.  We fix an orientation of $L$ that we denote by $\det$, i.e. an isomorphism $\det: \wedge^3_\Z L \to \Z$. Let $V=\Lambda \otimes_\Z \Q$ and write similarly $V_\R$ and $V_\C$ for the three--dimensional vector spaces over $\R$ and $\C$ obtained from $V$ by extending scalars; our fixed orientation on $L$ induces orientations on all these vector spaces that we still denote by $\det$.

Primitive elements in $\Lambda$ are in one-to-one correspondence with oriented planes in $L \otimes \Q$: the plane $H(a)$ corresponding to a primitive element $a \in \Lambda$ is the kernel of $a$, which divides $L$ into two subsets
$$L_+ (a) = \{ \alpha \in L \; \big| \; a (\alpha ) >0 \} \quad \mbox{and} \quad L_- (a) = \{ \alpha \in L \; \big| \; a (\alpha ) \leq 0 \}.$$
An ordered basis $\lambda , \mu$ of $H(a)$ is (positively) oriented if $\det (\lambda , \mu , \delta )>0$ whenever $\delta (a) >0$. A primitive element $a \in \Lambda$ also determines an open subset $U_a \subset V_\C$ defined by 
\begin{equation} \label{eq:def_U_a}
U_a = \left\{ z \in V_\C \; \big| \;  \mathrm{Im} \left(\lambda (z) \overline{\mu (z)} \right) >0 \right\}
\end{equation}
for any oriented basis $(\lambda,\mu)$ of $H(a)$.

Fix an ordered pair $(a, b)$ of linearly independent primitive elements of $\Lambda$. It determines oriented planes $H(a), H(b) 
 \subset L \otimes \Q$ that intersect in a line spanned by the element $\det (a , b , \cdot) \in \Lambda^\vee = L$. This element is not primitive in general, but we may write
\begin{equation}\label{defsgamma1}
\det (a,b , \cdot ) = s \gamma 
\end{equation}
for a  unique pair $(s,\gamma)$ consisting of a positive integer $s$ and a primitive element $\gamma$ of $L$.

To the pair $(a,b)$ we also associate the cones
\begin{equation}
\begin{split}
C_{\texttt{+-}} (a , b ) &= \{ \delta \in L\; \big| \; \delta (a) >0, \ \delta (b) \leq 0 \} \\
C_{\texttt{-+}} (a , b ) &= \{ \delta \in L\; \big| \; \delta (a) \leq  0 , \ \delta (b) > 0 \}.
\end{split}
\end{equation}
Following first works of Ruijsenaars \cite{Ruijsenaard} and Felder and Varchenko \cite{FelderVarchenko}, Felder, Henriques, Rossi and Zhu \cite{Felder} defined the {\it elliptic gamma function}
\begin{equation}\label{felder-v-def}
\Gamma_{a,b} (w , z ; L)  = \frac{\prod_{\delta \in C_{\texttt{+-}} (a , b) / \Z \gamma} \left(1-e^{-2\pi i (\delta (z) -w ) / \gamma (z)} \right)}{\prod_{\delta \in C_{\texttt{-+}} (a , b) / \Z \gamma} \left(1-e^{2\pi i (\delta (z) -w ) / \gamma (z)} \right)}\cdot
\end{equation}
The product \eqref{felder-v-def} converges to a meromorphic function on $\C \times (U_{a} \cap U_b )$. It has simple zeros at $w = \delta (z) + n \gamma (z)$, with $\delta \in C_{\texttt{+-}} (a , b )$, $n \in \Z$, and simple poles at $w = \delta (z) + n \gamma (z)$, with $\delta \in C_{\texttt{-+}} (a , b )$, $n \in \Z$ (\cite[Prop. 3.3]{Felder}).

It is also proved in \cite{Felder} that the function \eqref{felder-v-def} can be expressed in terms of the basic elliptic gamma function $\Gamma$ defined in the introduction.  For this, let $\alpha, \beta \in L$ be such that 
\begin{equation}\label{defalphabeta1}
\alpha (b) = \beta (a) = 0 \quad \mbox{and} \quad \alpha (a) >0, \quad \beta (b) >0.
\end{equation}
Then $(\alpha,\beta,\gamma)$ is an oriented basis of $L \otimes \Q$, hence $(\beta, \gamma)$ and $(\gamma,\alpha)$ are oriented bases of $H(a)$ and $H(b)$ respectively and so
\begin{equation}
U_{a} \cap U_b = \{ z \in V_\C \; \big| \; \mathrm{Im} (\alpha (z) / \gamma (z) ) <0 \mbox{ and } \mathrm{Im} (\beta (z) / \gamma (z) ) >0 \}.
\end{equation}
Then Proposition 3.5 of loc. cit. shows that
\begin{equation} \label{eq:Gamma_a_b_as_prod_over_F}
\Gamma_{a,b} (w , z ; L)  = \prod_{\delta \in F / \Z \gamma} \Gamma \left( \frac{w+\delta (z)}{\gamma (z)} , \frac{\alpha (z)}{\gamma (z)} , \frac{\beta (z)}{\gamma (z)} \right).
\end{equation}
Here the product is finite and $F$ is the set of $\delta \in L$ such that 
$$0 \leq \delta (a) < \alpha (a) \quad \mbox{and} \quad 0 \leq \delta (b) < \beta (b).$$

The group $\SL (\Lambda) \cong \SL_3 (\Z)$ acts on $\Lambda$, $V$, $V_\C$. If $a$ is a primitive vector in $\Lambda$ and $g$ an element in $\SL(\Lambda)$, then $ga$ is primitive, we have $U_{ga} = g U_a$, and 
\begin{equation} \label{eq:Gamma_a_b_equivariance}
\Gamma_{ga,gb} (w , gz ; L) = \Gamma_{a,b} (w , z ; L ), \quad z \in U_a \cap U_b .
\end{equation}
The group $\SL (\Lambda)$ also acts on $L$ by duality.

\subsection{Complex cubic fields and values of the elliptic gamma function} \label{subsection:complex_cubic_fields_and_elliptic_gamma_function}

Let $K$ be a complex cubic number field over $\Q$ and let $\mathcal{O}_K$ be its ring of integers. We denote by $\sigma_\R:K \hookrightarrow \R$ the unique real embedding of $K$ and we fix once and for all a complex embedding $\sigma_\C: K \hookrightarrow \C$ of $K$. The map $x \mapsto (\sigma_\R(x),\sigma_\C(x))$ induces an isomorphism  $K \otimes \R \simeq \R \times \C$ and hence an orientation on $K$ by pulling back the standard orientation on $\R \times \C$.

Let $\mathfrak{f}$ and $\mathfrak{b}$ be non-zero ideals of $\mathcal{O}_K$ with $(\mathfrak{f}, \mathfrak{b})=1$. Write $L=\mathfrak{fb}^{-1}$; it is a free $\Z$--module of rank three and a fractional ideal of $K$ and the unit $1_K \in K$ is a primitive $\mathfrak{f}$-division point of $L$. We denote by $q$ the smallest positive integer in $L \cap \Z$ (equivalently, in $\mathfrak{f} \cap \Z$), so that $q$ is the order of $1_K$ in $L \otimes \Q/L$. The group of units 
$$\mathcal{O}_{\mathfrak{f}}^{+ ,\times} = \{ u \in \mathcal{O}_K^\times \; \big| \; u - 1_K \in \mathfrak{f}  \mbox{ and } \sigma_\R(u) >0 \}$$ 
is of rank $1$, it acts by multiplication on the $3$-dimensional $\Q$-vector space $K$ and preserves the lattice $L$. We assume that this group equals $\mathcal{O}_K^\times \cap (1+\mathfrak{f})$, i.e. that there are no units in $\mathcal{O}_K^\times \cap (1+\mathfrak{f})$ of negative norm (in particular, this implies that $-1 \notin 1+\mathfrak{f}$ and hence that $\mathfrak{f} \neq \mathcal{O}_K$ and $q \geq 3$); equivalently, the real place of $K$ ramifies in the narrow ray class field of $K$ of conductor $\mathfrak{f}$. We fix the generator $\varepsilon \in \mathcal{O}_{\mathfrak{f}}^{+ ,\times}$ such that $\varepsilon_\R:=\sigma_\R(\varepsilon) \in (0,1)$.

We will now define some complex numbers $\mathbf{\Gamma}_{\mathfrak{a},h}(L)$ using the elliptic gamma function. As the notation indicates, these depend on $L$ as well as on some auxiliary data consisting of an ideal $\mathfrak{a} \subset \mathcal{O}_K$ and an element $h \in L \otimes \Q$. These data must satisfy the following conditions: for the ideal $\mathfrak{a}$ we require that
\begin{equation} \label{eq:conditions_on_aa}
\mathfrak{a} \text{ is relatively prime to }\mathrm{N}(\mathfrak{f}) \text{ and } \mathrm{N}(\mathfrak{a}) \text{ is an odd prime,}
\end{equation} 
while for $h$ we require that
\begin{equation} \label{eq:h_and_lambda}
h=\lambda/q
\end{equation}
for $\lambda$ an {\it admissible} element of $L$. Here we say that  $\lambda \in L$ is admissible if it satisfies the following congruence conditions:
\begin{itemize}
\item $\lambda/q$ is congruent to $1_K$ modulo $L$,
\item $\lambda/\mathrm{N}(\mathfrak{a})$ belongs to $\mathfrak{a}^{-1}L$ and generates the quotient $\mathfrak{a}^{-1}L/L$.
\end{itemize}
These conditions are equivalent to $\lambda \in \mathrm{N}(\mathfrak{a})\mathfrak{a}^{-1}L$ with $\lambda \equiv q \mod qL$ and $\lambda \not\equiv 0 \mod \mathrm{N}(\mathfrak{a})L$, and so the existence of admissible elements follows from the isomorphism 
$$
\mathrm{N}(\mathfrak{a})\mathfrak{a}^{-1}\mathfrak{f}/(q\mathrm{N}(\mathfrak{a}))\mathfrak{f} \simeq \mathfrak{f}/q\mathfrak{f} \times \mathrm{N}(\mathfrak{a})\mathfrak{a}^{-1}/(\mathrm{N}(\mathfrak{a})).
$$
We note that if $\lambda$ is an admissible element of $L$ and $u \in \mathcal{O}_\mathfrak{f}^{+, \times}$, then $u\lambda$ is an admissible element of $L$. If $\lambda$ is admissible for $L=\mathfrak{fb}^{-1}$ and $\alpha \in 1+\mathfrak{f}$, then $\alpha^{-1} \lambda$ is admissible for $\mathfrak{f}((\alpha)\mathfrak{b})^{-1}$.

Let us now fix $\mathfrak{a}$ and $h=\lambda/q$ with $\lambda \in L$ admissible as above. The restriction of our fixed complex embedding $\sigma_\C: K \to \C$ to $L$ defines a non-zero element 
\begin{equation}
{\bf x} \in V_\C = \mathrm{Hom} (L , \C)
\end{equation}
and therefore a point $\langle {\bf x} \rangle$ in 
\begin{equation}\mathbf{X} = ( V_\C - \C \cdot V_\R ) / \C^\times = \C \mathrm{P}^{2} - \R \mathrm{P}^{2}.
\end{equation}
It will be convenient to write write $\det: \wedge_\Z^3 L \to \Z$ for the isomorphism induced by the orientation of $K$ fixed above. The unit $\varepsilon$ acts by multiplication on $K$ and preserves $L$ and its  orientation. It therefore defines a linear map in $\SL (L)$ and, by duality, a linear map in $\SL(\Lambda)$. We loosely denote the latter by the same symbol $\varepsilon$. Note that the action of $\varepsilon$ on $V_\C$ preserves the complex line $\langle {\bf x} \rangle \in \mathbf{X}$. In fact, the vector ${\bf x}$ is an eigenvector of the linear map $\varepsilon$ with eigenvalue $\sigma_\C(\varepsilon)^{-1}$.

\begin{lemma} \label{L:L22}
Let $\lambda \in L$ be admissible. There exists a unique primitive element $a=a_\lambda \in \Lambda$ such that 
\begin{itemize}
\item[(i)] $\det (a , \varepsilon  a , \cdot ) \in \Lambda^\vee = L $ is a non-zero multiple of $\lambda$, and
\item[(ii)] ${\bf x} \in U_a$.
\end{itemize}
\end{lemma}
\begin{proof}
Since $\varepsilon$ has no rational eigenvalues, the elements $\lambda$ and $\varepsilon^{-1}\lambda \in L \otimes \Q$ are linearly independent. The condition that $\det(a,\varepsilon a, \cdot)=n \lambda$ for some scalar $n$ is equivalent to $a \in \ker(\lambda) \cap \ker(\varepsilon^{-1}\lambda)$, which has exactly two primitive solutions $\pm a \in \Lambda$. These correspond to the two orientations on the same plane $H(\pm a) \subset L \otimes \Q$, and only of them will satisfy $(ii)$. 
\end{proof}

Let us fix $\mathfrak{a}$ and $h$ satisfying the conditions above, write $a$ for the element of $\Lambda$ given by the lemma and define $b=\varepsilon a$. It follows from condition (i) in the Lemma that $a$ and $b$ both vanish on $\lambda$ and hence --- since $\mathfrak{a}^{-1}L = L + \Z \lambda/\mathrm{N}(\mathfrak{a})$ --- take integral values on $\mathfrak{a}^{-1} L$.  The linear forms $a$ and $b$ define primitive elements of $\mathrm{Hom}_\Z (\mathfrak{a}^{-1} L , \Z)$. The elliptic gamma function $\Gamma_{a,b} (w , z, \mathfrak{a}^{-1} L)$ is therefore well defined, and is given by the following formula. We write $\gamma=\gamma(a,b)$ for the primitive element in $L$ defined uniquely by the condition $\det(a, \varepsilon a, \cdot)=s \gamma$ for some $s \geq 1$ (cf. \eqref{defsgamma1}).

\begin{lemma}[Distribution relation]
We have 
\begin{equation*}
\Gamma_{a,b} (w , z ; \mathfrak{a}^{-1} L) = 
\prod_{k=0}^{\mathrm{N} (\mathfrak{a}) -1}  \Gamma_{a,b} \left( w + k\frac{\gamma (z)}{ \mathrm{N}(\mathfrak{a})} , z ; L \right).
\end{equation*}
\end{lemma}
\begin{proof}
Since $\lambda, \gamma \in L$ both belong to the line $\ker(a) \cap \ker(\varepsilon a)$ and $\gamma$ is primitive, we have $\lambda = n\gamma$ for some integer $n$; moreover $(n,\mathrm{N}(\mathfrak{a}))=1$ since $\lambda/\mathrm{N}(\mathfrak{a}) \notin L$. Writing $nx+ \mathrm{N}(\mathfrak{a})y=1$ for some integers $x$ and $y$, this shows that $x\lambda/\mathrm{N}(\mathfrak{a})=\gamma/\mathrm{N}(\mathfrak{a})-y \gamma$ and hence that 
$$
\mathfrak{a}^{-1}L=L+\Z \lambda/\mathrm{N}(\mathfrak{a}) = L+\Z \gamma/\mathrm{N}(\mathfrak{a}).
$$

The identity now follows from the definition \eqref{felder-v-def}: let us complete $\gamma$ to a basis $\{\lambda_1,\lambda_2,\gamma\}$ of $L$. Then $\{\lambda_1, \lambda_2, \gamma/\mathrm{N}(\mathfrak{a})\}$ is a basis of $\mathfrak{a}^{-1}L$, and the map sending a vector $a \lambda_1 + b \lambda_2 +c \gamma \in L$ to $a \lambda_1 + b \lambda_2 +c \gamma/\mathrm{N}(\mathfrak{a}) \in \mathfrak{a}^{-1}L$ induces a bijection between $C_{{\texttt{+-}}} (a,b)$ and the set of vectors $\delta' \in  \mathfrak{a}^{-1} L$ with $\delta ' (a) >0$ and $\delta' (b ) \leq 0$. Classes modulo $\Z \gamma$ correspond to classes modulo $\Z \gamma / \mathrm{N} (\mathfrak{a})$. We therefore get the following relation:
\begin{equation*}
\begin{split}
\prod_{k=0}^{\mathrm{N} (\mathfrak{a}) -1}  \Gamma_{a,b} \left( w + k\frac{\gamma (z)}{ \mathrm{N} (\mathfrak{a})} , z  ; L \right)
& = \frac{\prod_{\delta \in C_{\texttt{+-}} (a , b) / \Z \gamma} \left(1-e^{-2\pi i \mathrm{N} (\mathfrak{a}) (\delta (z) -w ) / \gamma (z)} \right)}{\prod_{\delta \in C_{\texttt{-+}} (a , b) / \Z \gamma} \left(1-e^{2\pi i \mathrm{N} (\mathfrak{a}) (\delta (z) -w ) / \gamma (z)} \right)} \\
& = \Gamma_{a,b} (w , z ; \mathfrak{a}^{-1} L).
\end{split}
\end{equation*}
\end{proof}

We may then define the \emph{smoothed elliptic gamma function} to be the ``quotient of quotients''\footnote{cf. Eisenstein \cite{Eisenstein}, see appendix.}
\begin{equation} \label{eq:def_smoothed_ell_gamma_function}
\mathbf{\Gamma}_{\mathfrak{a} , h} (w , z ; L )  =  \frac{\Gamma_{a, b} (w , z ; \mathfrak{a}^{-1} L )}{\Gamma_{a,b} (w,z ; L)^{\mathrm{N} (\mathfrak{a})}}.
\end{equation}

It converges to a meromorphic function on $\C \times (U_{a} \cap U_b )$. Note that, since the elliptic gamma function $\Gamma_{a,b}$ in \eqref{felder-v-def} satisfies $\Gamma_{b,a} = \Gamma_{a,b}^{-1}$, we may also write
\begin{equation*}
\mathbf{\Gamma}_{\mathfrak{a} , h} (w , z ; L ) = \frac{\Gamma_{b,\varepsilon^{-1} b} (w,z ; L)^{\mathrm{N} (\mathfrak{a})}}{\Gamma_{b, \varepsilon^{-1}b} (w , z ; \mathfrak{a}^{-1} L )}.
\end{equation*}

The notation here is to recall that $\mathbf{\Gamma}_{\mathfrak{a} , h} (w , z ; L )$ depends only on the choice of $\mathfrak{a}$ and $h$ and (oriented) lattice $L=\mathfrak{fb}^{-1}$.

Finally, note that since ${\bf x}$ belongs to $U_a$ and $\mathbf{x}$ is an eigenvector of $\varepsilon$, we also have $\mathbf{x} \in U_b$. Thus we can evaluate $\mathbf{\Gamma}_{\mathfrak{a},h}(w,z;L)$ at $z=\mathbf{x}$. Since $\mathfrak{a}$ satisfies \eqref{eq:conditions_on_aa}, we have $(\mathrm{N}(\mathfrak{a}),q)=1$ and hence $h=\lambda/q \notin \mathfrak{a}^{-1}L$. From the description of the divisor of $\Gamma_{a,b}$ below \eqref{felder-v-def} we conclude that the complex number $w=h(\mathbf{x})$ lies away from the zeroes and poles of $\mathbf{\Gamma}_{\mathfrak{a},h}(w,\mathbf{x};L)$. We arrive at our main definition.
\begin{definition}
\begin{equation*}
\mathbf{\Gamma}_{\mathfrak{a},h}(L) = \mathbf{\Gamma}_{\mathfrak{a},h}(h(\mathbf{x}),\mathbf{x};L) \in \C^\times.
\end{equation*}
\end{definition}

Below we formulate a precise conjecture stating that under certain conditions the numbers $\mathbf{\Gamma}_{\mathfrak{a},h}(L)$ are independent of the choice of $h$. For now we limit ourselves to the following remarks that follow from the equivariance property \eqref{eq:Gamma_a_b_equivariance} and the discussion after \eqref{eq:h_and_lambda}.

\medskip
\noindent
{\it Remarks.} Fix a lattice $L=\mathfrak{fb}^{-1}$ as above and an ideal $\mathfrak{a}$ satisfying \eqref{eq:conditions_on_aa}.
\begin{enumerate}[leftmargin=*]
\item Let $u \in \mathcal{O}_\mathfrak{f}^{+,\times}$ and $h=\lambda/q$ with $\lambda$ admissible for $L$. Then $u\lambda$ is admissible for $L$. The element $a=a_\lambda \in \Lambda$ defined by Lemma \ref{L:L22} satisfies $a_{u\lambda}=ua_\lambda$ and hence, since $\mathbf{x}$ is an eigenvector of $u$, we have
$$
\mathbf{\Gamma}_{\mathfrak{a},uh}(L)=\mathbf{\Gamma}_{\mathfrak{a},h}(L).
$$
\item Suppose that $\alpha \in 1+\mathfrak{f}$ satisfies $\sigma_\R(\alpha)>0$. Since $\mathrm{N}(\alpha)=\sigma_\R(\alpha)|\sigma_\C(\alpha)|^2$, multiplication by $\alpha$ respects the orientation on $K$. Let $h=\lambda/q$ with $\lambda$ admissible for $L$; then $\alpha^{-1}h=(\alpha^{-1}\lambda)/q$ with $\alpha^{-1} \lambda$ admissible for $(\alpha)^{-1}L$. The primitive element $a_{\alpha^{-1}\lambda}$ in $\mathrm{Hom}((\alpha)^{-1}L,\Z)$ provided by Lemma \ref{L:L22} equals $\alpha^{-1}a_\lambda$. Since $\mathbf{x}$ is an eigenvector of $\alpha$, it follows that
\begin{equation} \label{eq:Gamma_invariance_P_+_ff}
\mathbf{\Gamma}_{\mathfrak{a},\alpha^{-1}h}((\alpha)^{-1}L)=\mathbf{\Gamma}_{\mathfrak{a},h}(L), \quad \alpha \in 1+ \mathfrak{f}, \quad \sigma_\R(\alpha)>0.
\end{equation}

\item Suppose that $\alpha \in -1+\mathfrak{f}$ satisfies $\sigma_\R(\alpha)>0$. Let $h=\lambda/q$ with $\lambda$ admissible for $L$; then $-\alpha^{-1}\lambda$ is admissible for $(\alpha)^{-1}L$. Arguing as above, we find that the primitive element associated with $-\alpha^{-1}\lambda$ provided by Lemma \ref{L:L22} is $\alpha^{-1}a_\lambda$. Hence we have $\gamma(\alpha^{-1}a,\alpha^{-1}b)=\alpha^{-1}\gamma(a,b)$ for the vector determined by the wedge $(\alpha^{-1}a,\alpha^{-1}b)$, and since $\mathbf{x}$ is an eigenvector of $\alpha$ we find that
\begin{equation}
\label{eq:Gamma_invariance_P_ff}
\mathbf{\Gamma}_{\mathfrak{a},-\alpha^{-1}h}((\alpha)^{-1}L) = \mathbf{\Gamma}_{\mathfrak{a},h}(-h(\mathbf{x}),\mathbf{x};L).
\end{equation}
We will show in \eqref{remark:evaluation_-h} that up to a third root of unity (trivial when $\mathfrak{a}$ is coprime to $6$), this number equals $\mathbf{\Gamma}_{\mathfrak{a},h}(L)^{-1}$.

\item Let us write temporarily $\mathbf{\Gamma}_{\mathfrak{a},h}(L,\sigma_\C)$ for $\mathbf{\Gamma}_{\mathfrak{a},h}(L)$ to emphasise the role of the fixed complex embedding $\sigma_\C$. Fixing $h=\lambda/q$ with $\lambda$ admissible, upon replacing $\sigma_\C$ by its complex conjugate $\overline{\sigma}_\C$, the orientation of $K$ changes and so the element $a_\lambda$ associated with $\lambda$ provided by Lemma \ref{L:L22} stays the same. Hence so do the elements $\alpha$ and $\beta$ in \eqref{defalphabeta1} and the set $F$, while the vector $\gamma=\gamma(a,\varepsilon a)$ in \eqref{defsgamma1} changes sign. Thus replacing $\sigma_\C$ by $\overline{\sigma}_\C$ changes the quantities $\tau=\alpha(\mathbf{x})/\gamma(\mathbf{x})$ and $\sigma=\beta(\mathbf{x})/\gamma(\mathbf{x})$ to $-\overline{\tau}$ and $-\overline{\sigma}$. Using \eqref{eq:Gamma_a_b_as_prod_over_F} and the identity $\Gamma(-\overline{z},-\overline{\tau},-\overline{\sigma})=\overline{\Gamma(z,\tau,\sigma)}$, we find that
$$
\mathbf{\Gamma}_{\mathfrak{a},h}(L,\overline{\sigma}_\C)= \overline{\mathbf{\Gamma}_{\mathfrak{a},h}(L,\sigma_\C)}.
$$
\end{enumerate}

\subsection{The main conjecture} \label{S:conj}

Let $K (\mathfrak{f})$ be the narrow ray class field of $K$ of conductor $\mathfrak{f}$. It is a totally complex abelian extension of $K$ that is maximal among abelian extensions ramified only at $\mathfrak{f}$. In general there is no known easy way to construct it. Elliptic units, which are obtained by evaluating modular units at quadratic imaginary arguments of $\mathcal{H}$, allow for the analytic construction of abelian extensions of imaginary quadratic fields. The following conjecture aims to transpose the theory of elliptic units to the context of complex cubic fields. 

Denote by $I (\mathfrak{f})$ the group of fractional ideals of $K$ generated by ideals coprime to $\mathfrak{f}$, and by $P(\mathfrak{f})$ and $P(\mathfrak{f})_+$ the subgroups of principal fractional ideals $(\alpha)$ generated by elements $\alpha \in K^\times$ such that $v_\mathfrak{p}(\alpha-1) \geq v_\mathfrak{p}(\mathfrak{f})$ for every prime $\mathfrak{p}$ dividing $\mathfrak{f}$, requiring furthermore $\sigma_\R(\alpha) >0$ for $(\alpha)$ in $P(\mathfrak{f})_+$. 
 
According to class field theory, the narrow ray class group
$$\mathrm{Cl}_K^+ (\mathfrak{f}) = I (\mathfrak{f}) / P(\mathfrak{f})_+$$
is isomorphic to the Galois group $\mathrm{Gal} (K (\mathfrak{f}) /K)$ via the Artin map. We denote by $\sigma_\mathfrak{c} = (\mathfrak{c} , K (\mathfrak{f}) /K)$ in $\mathrm{Gal} (K (\mathfrak{f}) /K)$ the element that this map assigns to an ideal $\mathfrak{c}$ in $I (\mathfrak{f})$, normalised so that for a prime ideal $\mathfrak{p}$ we have $\sigma_\mathfrak{p}(x) \equiv x^{\mathrm{N}(\mathfrak{p})} \mod \mathfrak{P}$ when $\mathfrak{p}=\mathfrak{P} \cap \mathcal{O}_K$.

\medskip
\noindent
{\bf Conjecture.} {\it Let $\mathfrak{f}$ be an ideal of $\mathcal{O}_K$ such that $K(\mathfrak{f})$ is totally complex and fix a complex embedding $K(\mathfrak{f}) \hookrightarrow \C$ extending the fixed embedding $K \hookrightarrow \C$. Let $\mathfrak{b}$ be an ideal of $\mathcal{O}_K$ with $(\mathfrak{f},\mathfrak{b})=1$ and set $L=\mathfrak{fb}^{-1}$. Suppose that $\mathfrak{a} \subset \mathcal{O}_K$ is coprime to $6 \mathrm{N}(\mathfrak{f})$
and that $\mathrm{N}(\mathfrak{a})$ is an odd prime and let $h=\lambda/q$ with $\lambda \in L$ admissible and $q$ the smallest positive integer in $\mathfrak{f} \cap \Z$.
\begin{enumerate}
\item The number $\mathbf{\Gamma}_{\mathfrak{a} , h} ( L)$ only depends on $L$ and $\mathfrak{a},$ and is the image in $\C$ of a unit $\mathbf{u}_{L, \mathfrak{a}} \in \mathcal{O}_{K (\mathfrak{f})}^\times$.
\item If $w$ is an  archimedean place of $K (\mathfrak{f})$ above the real embedding of $K$, then $|\mathbf{u}_{L, \mathfrak{a}}|_w =1$.
\item For any $\mathfrak{c} \in I(\mathfrak{f})$ we have the \emph{reciprocity law}
$$\sigma_{\mathfrak{c}} (\mathbf{u}_{L, \mathfrak{a}}) = \mathbf{u}_{\mathfrak{c}^{-1} L, \mathfrak{a}}.$$
\end{enumerate}
}
\medskip

\noindent
{\it Remarks.} 
\begin{enumerate}[leftmargin=*]
    \item By \eqref{eq:Gamma_invariance_P_+_ff}, part (1) of the conjecture implies that the number $\mathbf{\Gamma}_{\mathfrak{a},h}(\mathfrak{fb}^{-1})$ only depends on $\mathfrak{b}$ through its class in the narrow ray class group $\mathrm{Cl}^+_K(\mathfrak{f})$. This is consistent with part (3). 

\item Since $\mathrm{Gal}(K(\mathfrak{f})/K)$ acts transitively on the places of $K(\mathfrak{f})$ above the complex place of $K$, the valuation $|\mathbf{u}_{L,\mathfrak{a}}|_w$ for such places is determined by (3) and Theorem \ref{T:KLintro} below.

\item Let $\tau \in \mathrm{Gal}(K(\mathfrak{f})/K)$ correspond to the unique non-trivial element of the order two subgroup $P(\mathfrak{f})/P(\mathfrak{f})_+$ of $\mathrm{Cl}^+_K(\mathfrak{f})$. By \eqref{eq:Gamma_invariance_P_ff}, we have
$$
\tau(\mathbf{u}_{L,\mathfrak{a}}) =\mathbf{u}_{L,\mathfrak{a}}^{-1}.
$$

\item When $\mathfrak{a}$ is not coprime to $6$ we still conjecture that $\mathbf{\Gamma}_{\mathfrak{a} , h} ( L)$ is algebraic but in general we have to raise to a power to get a number that is the image in $\C$ of a unit in $K (\mathfrak{f})$. In all numerical examples we have checked, taking the 12th power suffices, and we find that the resulting unit only depends on $L$ and $\mathfrak{a}$ \emph{up to a root of unity in $K (\mathfrak{f})$}.

\end{enumerate}

\subsection{A numerical example} \label{subsection:first_numerical_example}

Stark's conjecture \cite{Stark,Tate}  --- that we discuss a bit more below --- posits a remarkable connection between $L$-function values at $s=0$ and units in number fields. 
To the authors' knowledge the first example for complex cubic fields was computed by Samit Dasgupta in his Harvard senior thesis \cite{Dasgupta}. More examples were later studied in \cite{DummitAl}. 

We consider here Dasgupta's original example. We give a larger collection of numerical examples at the end of the paper. 

Let $K= \Q (\beta)$, where $\beta \in \C$ is the root of $x^3-x^2+5x+1$ with positive imaginary part. This is the complex cubic field of smallest discriminant (in absolute value) with class number $3$. 

The discriminant of $K$ over $\Q$ is $-588=-2^2 \cdot 3 \cdot 7^2$. We have
$$(3) = \mathfrak{p}_1 \mathfrak{p}_2^2 \quad \mbox{where} \quad \mathfrak{p}_1 = (3 , \beta +1) \quad \mbox{and} \quad \mathfrak{p}_2 = (3 , \beta -1).$$
Take $\mathfrak{f}=\mathfrak{p_1}$. The corresponding narrow ray class field $K(\mathfrak{f})=K(\mathfrak{p_1} \infty_{\R})$ has relative degree $6$ over $K$. It is the compositum of $K(\sqrt{-3})$ and the Hilbert class field $H=K(\theta)$, where $\theta = \zeta_7 + \zeta_7^{-1}$ satisfies the equation $\theta^3+\theta^2-2\theta-1=0$. The Galois group $G$ of the abelian extension $K(\mathfrak{f})/K$ is cyclic of order $6$. Also $K(\mathfrak{f})$ possesses  $6$ roots of unity. 
We fix a complex embedding of $K(\mathfrak{f})$ extending that of $K.$

After many clever convoluted numerical computations, Dasgupta found a unique possible Stark unit $\mathbf{u}_{\rm Stark}$, up to a choice of root of unity, and checked it has the desired valuation at the complex places to at least $25$ decimal places. This led him to raise plainly the crucial question: 

\begin{quote}
Consider the Stark conjecture [$\ldots$] Is there a formula for the actual [complex] value of a (conjectured) Stark unit $\mathbf{u}_{\rm Stark}$? In other words, can one propose a putative formula for the \textit{argument} of a Stark unit $\mathbf{u}_{\rm Stark}$ in addition to the Stark formula for its absolute value?

Not only would answering this question make numerical confirmations easier [$\ldots$], but it would also enable one to make a reference to Hilbert's 12th problem in the complex $v$ case, as one can do in the real $v$ case.
\end{quote}

Our conjecture provides a tentative solution to this question. Let our smoothing ideal $\mathfrak{a}$ be the unique ideal with norm $5$. When $\mathfrak{b}$ varies among ideals that are coprime to both $\mathfrak{f}$ and $\mathfrak{a},$ we conjecture that $\mathbf{\Gamma}_{\mathfrak{a} , h} ( \mathfrak{f}\mathfrak{b}^{-1} )$ is the complex image of a unit $\mathbf{u}_{\mathfrak{fb}^{-1},\mathfrak{a}}$ in $K (\mathfrak{f})$ that is independent of the particular choice of $h$. Moreover, the action of $G$ on these units is (conjecturally) given by the rule $\mathbf{u}_{ \mathfrak{f}\mathfrak{b}^{-1}, \mathfrak{a}} = \sigma_{\mathfrak{b}} (\mathbf{u}_{ \mathfrak{f}, \mathfrak{a}})$.

The class of the prime ideal $\mathfrak{b} = (2, \beta -1)$, lying above $2$, generates the ray class group. Identifying the alleged units with their complex images, we compute (to $1000$ digits in a few seconds):\footnote{We numerically verify that changing the choice of $h$ does not affect the result.}
\begin{equation*}
\begin{array}{lcccc}
\mathbf{\Gamma}_{\mathfrak{a} , h} ( \mathfrak{f})^{-1} & \approx & -1.3795863226 \ldots + \mathrm{i} \cdot  2.0250077123 \ldots  & \approx & \mathbf{\Gamma}_{\mathfrak{a} , h} ( \mathfrak{f}\mathfrak{b}^{-3} )   \\
\mathbf{\Gamma}_{\mathfrak{a} , h} ( \mathfrak{f}\mathfrak{b}^{-1} )^{-1} & \approx & 1.3269203008 \ldots - \mathrm{i} \cdot  1.2639106201 \ldots & \approx & \mathbf{\Gamma}_{\mathfrak{a} , h} ( \mathfrak{f}\mathfrak{b}^{-4} )\\
\mathbf{\Gamma}_{\mathfrak{a} , h} ( \mathfrak{f}\mathfrak{b}^{-2} )^{-1}  & \approx & -4.8390562074 \ldots - \mathrm{i} \cdot  7.5167566542 \ldots & \approx  & \mathbf{\Gamma}_{\mathfrak{a} , h} ( \mathfrak{f}\mathfrak{b}^{-5} ).
\end{array}
\end{equation*}
These units coincide with the roots of 
$$x^6 + (3+3\beta ) x^5 + (2+ \beta - \beta^2) x^4 + (1+3\beta + 7 \beta^2) x^3 + (2+ \beta - \beta^2) x^2 + (3+3\beta ) x +1$$
to $1000$ decimal digits. Using  Pari-GP \cite{parigp} one can verify that this polynomial is irreducible over $K$ and has  its roots in $K(\mathfrak{f})$.
Furthermore, being monic with constant term equal to one, any root $z$ is a unit. The unique element $\tau \in G$ of order two acts on the roots by $\tau (z) = z^{-1}$. This implies that $|z|_w=1$ for every place $w$ of $K(\mathfrak{f})$ lying over the real place of $K$. The above values verify numerically our reciprocity law for the Galois action, in agreement with our conjecture.

According to Dasgupta's computation the Stark unit $\mathbf{u}_{\rm Stark}$ (if it exists) must be --- after possibly rescaling by a root of unity --- a root of
\begin{equation*}
\begin{split}
x^6 &- (\beta^2 -4 \beta +1) x^5 + (-7\beta^2 +10 \beta +5) x^4 \\
&- (22 \beta^2 + 17 \beta + 6) x^3 + (-7\beta^2 +10 \beta +5) x^2 - (\beta^2 -4 \beta +1) x +1.
\end{split}
\end{equation*}
We choose the complex embedding of $K(\mathbf{u}_{\rm Stark})$ mapping $\mathbf{u}_{\rm Stark}$ to the root whose first decimal digits are $ -0.256\ldots - \mathrm{i}  \cdot  0.077\ldots$. We then verify numerically that 
$$\mathbf{u}_{\rm Stark}^{\mathrm{N} (\mathfrak{a})-\sigma_{\mathfrak{a}}} \approx \mathbf{\Gamma}_{\mathfrak{a} , h} ( \mathfrak{f} )^{6},$$ 
where $\approx$ represents an error less than $10^{-1000}$. Note that, even though the Stark unit $\mathbf{u}_{\rm Stark}$ is only defined up to a root of unity, its power $\mathbf{u}_{\rm Stark}^{\mathrm{N} (\mathfrak{a})-\sigma_{\mathfrak{a}}}$ is uniquely defined. Since $\mathfrak{a}$ is of norm $5$, the element ${\mathrm{N} (\mathfrak{a})-\sigma_{\mathfrak{a}}} \in \Z [G]$ indeed annihilates all roots of unity of $K(\mathfrak{f})$ (see Remark below).

We could have taken another smoothing ideal, e.g. $\mathfrak{a} = (2, \beta -1)$, the prime ideal lying above $2$. Then $(\mathfrak{a} , 6) \neq 1$. For this choice of $\mathfrak{a}$, we can numerically compute $\mathbf{\Gamma}_{\mathfrak{a} , h} (  \mathfrak{f})$. The \texttt{algdep} command of Pari-GP allows to show that this complex number is still very close to the image of an algebraic unit but that unit does not belong to $K(\mathfrak{f})$ and its value depends on the choice of $h$. In that case we need to raise this number to the power $4$ to get the image of a unit in $K(\mathfrak{f})$. This is coherent with our conjecture.

\medskip

\noindent
{\it Remark.} 
Let $K$ be a complex cubic field. Suppose that the narrow ray class field $K(\mathfrak{f})$ contains  a $p$-th root of unity for a prime $p$ coprime to $\mathfrak f$. Then the ramification degree (over $\Q$) of a prime above $p$ is at least $p-1$. The absolute ramification degree of such a prime in $K(\mathfrak{f})$ agres with the one of the prime of $K$ below it.  Because $[K:\Q]=3$, it follows that $p \leq 3$. Thus, by \cite[Lemme IV.1.1]{Tate}, when $(\mathfrak{a},6 \mathfrak{f})=1$ the element $\mathrm{N}(\mathfrak{a})-\sigma_\mathfrak{a} \in \Z[\mathrm{Gal}(K(\mathfrak{f})/K)]$ annihilates the roots of unity.

\section{A limit formula for Hecke $L$-functions of complex cubic fields}

We keep notation as in the preceding section.

\subsection{Partial zeta functions}

Fix a class $C$ in $\mathrm{Cl}_K^+ (\mathfrak{f})$ and let $\sigma$ be the corresponding element in  $\mathrm{Gal} (K(\mathfrak{f})/K)$. Associated to these data is the partial zeta function
\begin{equation}
\zeta_{\mathfrak{f}} (C, s) = \zeta_{\mathfrak{f}} (\sigma , s) = \sum_{\mathfrak{c} \in C} \mathrm{N} (\mathfrak{c})^{-s}, \quad \mathrm{Re}(s) >1,
\end{equation}
where the sum runs over all integral ideals $\mathfrak{c}$ in the ideal class $C$. Writing $C=[\mathfrak{b}]$ for an integral ideal $\mathfrak{b}$ of $\mathcal{O}_K$ and $L=\mathfrak{fb}^{-1}$, we have
$$
C = \{ \mathfrak{b} \mu \; : \; \mu \in (1+L)^+/ \mathcal{O}_\mathfrak{f}^{+ , \times} \},
$$
where $(1 + L)^+$ are the totally positive elements in $1+L$. 
The associated partial zeta function is therefore given for $\mathrm{Re}(s)>1$ by
\begin{equation*}
\zeta_{\mathfrak{f}} (\mathfrak{b} , s) = \mathrm{N} (\mathfrak{b})^{-s}  \underbrace{\sum_{\mu \in (1+L)^+/ \mathcal{O}_\mathfrak{f}^{+ , \times}} \mathrm{N} (\mu )^{-s}.}_{ \zeta (1_K , L , s)}
\end{equation*}
We have the identity 
\begin{equation} \label{E:partialsum}
\zeta (1_K , L , s) = \frac{1}{2} \left( \sum_{\mu \in  (1 + L) / \mathcal{O}_\mathfrak{f}^{+ , \times}} \frac{1}{|\mathrm{N} (\mu )|^{s}} + \sum_{\mu \in  (1 + L) / \mathcal{O}_\mathfrak{f}^{+ , \times}} \frac{\mathrm{sgn}(\mu_{\R})}{|\mathrm{N} (\mu )|^{s}} \right).
\end{equation}
This sum can be expressed as a linear combination of Hecke $L$-functions.

\subsubsection*{Hecke $L$-functions}

Let $\chi$ be a character of $\mathrm{Cl}_K^+ (\mathfrak{f})$.  The restriction of $\chi$ to $P(\mathfrak{f}) / P(\mathfrak{f})_+ \to \C^\times$ is of order $1$ or $2$. We say that $\chi$ is \emph{even} in the first case, and \emph{odd} otherwise. For $\mathrm{Re}(s) >1$, we define an $L$-function
\begin{equation*}
L_\mathfrak{f} (\chi ,s ) = \sum_{\substack{\mathfrak{a} \subset \mathcal{O}_K \\ (\mathfrak{a} , \mathfrak{f}) =1}} \frac{\chi (\mathfrak{a} )}{\mathrm{N} (\mathfrak{a})^s} = \prod_{(\mathfrak{p} , \mathfrak{f})=1} \left( 1 - \frac{\chi (\sigma_\mathfrak{p} )}{\mathrm{N} (\mathfrak{p})^s} \right)^{-1}.
\end{equation*}
It admits meromorphic continuation and satisfies a functional equation with $\Gamma$-factors 
$$\Gamma (\chi ,s ) = \left\{ \begin{array}{ll}
\Gamma_\R (s) \Gamma_\C (s) & \mbox{ is } \chi \mbox{ is even}, \\
\Gamma_\R (s+1) \Gamma_\C (s) & \mbox{ is } \chi \mbox{ is odd}.
\end{array} \right.$$
The functional equation implies that for $\chi \neq \chi_0$ (the trivial character), the completed $L$-function  $\Gamma(\chi,s)L_\mathfrak{f}(\chi,s)$ is entire; in particular $L_\mathfrak{f}(\chi,s)$ has at least a double zero at $s=0$ if $\chi \neq \chi_0$ is even.  When $\chi =\chi_0$, the function 
$$L_\mathfrak{f} (\chi_0 , s ) = \zeta_{K , \mathfrak{f}} (s) = \zeta_K (s) \prod_{\mathfrak{p} | \mathfrak{f}} (1- \mathrm{N}(\mathfrak{p})^{-s})$$
has analytic continuation to all of $\C$ with exactly one simple pole at $s=1$ and, since $\mathfrak{f} \neq \mathcal{O}_K$, it again has a zero of order $\geq 2$ at $s=0$. 
It follows that $L_\mathfrak{f} ' (\chi , 0) = 0$ if $\chi$ is even. 

\subsubsection*{Meromorphic continuation of partial $L$-functions}

The vanishing of $L'_\mathfrak{f}(\chi,0)$ for $\chi$ even gives the following expression for partial zeta functions.

\begin{lemma} \label{lemma:derivative_zeta_sign}
We have
\begin{equation*} 
\zeta' (1_K , L , 0) = \frac12 \frac{d}{ds} \left( \sum_{\mu \in (1 + L) / \mathcal{O}_\mathfrak{f}^{+ , \times}} \frac{\mathrm{sign}(\mu_{\R})}{|\mathrm{N} (\mu )|^{s}} \right)_{s=0} .
\end{equation*}
\end{lemma}
\begin{proof} The first sum in the right hand side of \eqref{E:partialsum} is
$$
\sum_{\mu \in (1+L)/\mathcal{O}_{\mathfrak{f}}^{+,\times}} \frac{1}{|\mathrm{N}(\mu)|^s} = \sum_{\mu \in (1+L)/\mathcal{O}_{\mathfrak{f}}^{\times}} \frac{1}{|\mathrm{N}(\mu)|^s}.
$$
This is a linear combination of $L$-functions $L_\mathfrak{f} (\chi , s)$ with $\chi$ even: if $\mathfrak{a}_1 , \ldots , \mathfrak{a}_t$ are integral ideals representing the classes in $\mathrm{Cl}_K(\mathfrak{f}) = I(\mathfrak{f}) / P (\mathfrak{f})$, then
\begin{equation*}
\begin{split}
L_\mathfrak{f} (\chi , s)  = \sum_{j=1}^t \frac{\chi (\mathfrak{a}_j)}{\mathrm{N} (\mathfrak{a}_j)^s} \sum_{\mu  \in (1+ \mathfrak{f a}_j^{-1}) / \mathcal{O}_\mathfrak{f}^{\times}}  \frac{1}{|\mathrm{N} (\mu )|^{s}}
\end{split}
\end{equation*}
and hence
\begin{equation*}
\mathrm{N}(\mathfrak{a}_j)^{-s} \sum_{\mu \in (1+\mathfrak{fa}_j^{-1})/\mathcal{O}_\mathfrak{f}^\times} \frac{1}{|\mathrm{N}(\mu)|^s} = \frac{1}{t} \sum_{\chi: \mathrm{Cl}_K(\mathfrak{f}) \to \C} \chi(\mathfrak{a}_j)^{-1}L_\mathfrak{f}(\chi,s).
\end{equation*}
Since $L_\mathfrak{f} ' (\chi , 0) =0$  if $\chi$ is even, the first sum in the right hand side of \eqref{E:partialsum} does not contribute to the derivative at $0$ and the lemma follows.  
\end{proof}

\subsection{A Kronecker limit formula}

Let $\mathfrak{a}$, $\mathfrak{f}$ and $\mathfrak{b}$ be ideals of $\mathcal{O}_K$ as in Section \ref{subsection:complex_cubic_fields_and_elliptic_gamma_function}. Consider the smoothed partial zeta function
\begin{equation} \label{E:smoothing}
\zeta_{\mathfrak{f} , \mathfrak{a}} (\mathfrak{b} , s) =  \zeta_{\mathfrak{f}} (\mathfrak{ab} , s) -\mathrm{N} (\mathfrak{a}) \zeta_{\mathfrak{f}} (\mathfrak{b} , s) .
\end{equation}

The next theorem is strongly reminiscent of the limit formulas expressing the derivative at $s=0$ of Hecke $L$-functions of imaginary quadratic fields in terms of logarithms of absolute values of elliptic units.

\begin{theorem}  \label{T:KLintro}
The modulus of $\mathbf{\Gamma}_{\mathfrak{a}, h} ( \mathfrak{fb}^{-1})$ is independent of $h$ and 
$$\zeta_{\mathfrak{f} , \mathfrak{a}} ' (\mathfrak{b} , 0) = \log \left| \mathbf{\Gamma}_{\mathfrak{a}, h}( \mathfrak{fb}^{-1}) \right|^2 .$$
\end{theorem}

\medskip
\noindent
{\it Remark.} Assuming the conjecture holds and denoting by $v$ the complex embedding of $K(\mathfrak{f})$ fixed in the statement of the conjecture, we have\footnote{Since $v$ is complex, the valuation $|\cdot |_v$ is the square of the modulus of the image by the complex embedding associated to $v$.} 
\begin{equation*} 
\zeta_{\mathfrak{f} , \mathfrak{a}} ' (\mathfrak{b} , 0) =   \log | \mathbf{u}_{L,\mathfrak{a}}|_v.
\end{equation*}

\medskip

\noindent 
{\it Example.} In \cite{Dasgupta} Dasgupta computes the special values of the $6$ partial zeta functions $\zeta_{\mathfrak{f}} (\mathfrak{b}^k , s)$ ($k \in \{0 , \ldots , 5\}$) when  
$$K=\Q (\beta) \quad \mbox{s.t.} \quad \beta^3 - \beta^2 +5 \beta +1 =0, \quad \mathfrak{f} = (3, \beta +1) \quad \mbox{and} \quad \mathfrak{b} = (2 , \beta -1).$$
His table reads  \cite[p. 43]{Dasgupta}:
$$
\begin{array}{lclcl}
\zeta_{\mathfrak{f}} ' (\mathfrak{b}^0 , 0) & = & -\zeta_{\mathfrak{f}} ' (\mathfrak{b}^3 , 0) & \approx & 0.4384785858524408926911022\ldots,\\
\zeta_{\mathfrak{f} } ' (\mathfrak{b} , 0) & = & -\zeta_{\mathfrak{f} , } ' (\mathfrak{b}^4 , 0) & \approx & 0.3999812299583346413364528\ldots,\\
\zeta_{\mathfrak{f} } ' (\mathfrak{b}^2 , 0) & = & -\zeta_{\mathfrak{f} } ' (\mathfrak{b}^5 , 0) & \approx & 0.7885047438188618467230391\ldots,
\end{array}
$$
and we confirm using the numerical values in \ref{subsection:first_numerical_example} that, for each value of $k, $
$$\log |\mathbf{\Gamma}_{\mathfrak{a},h}(\mathfrak{fb}^{-k})|^2 \approx \zeta_{\mathfrak{f} } ' (\mathfrak{b}^{k+1} , 0) - 5 \zeta_{\mathfrak{f} } ' (\mathfrak{b}^k , 0)   ,$$
where $\approx$ represents an error less than $10^{-25}$. Since $\mathfrak{a}$ and $\mathfrak{b}$ belong to the same class in the ray class group, this formula agrees with Theorem \ref{T:KLintro}.

\subsection{Relation with Stark's conjecture} \label{section:relation_with_Starks_conjecture}

The following is a particular case of Stark's celebrated conjectures \cite{Stark,Tate} that generalize and refine the Dirichlet class number formula. Let $e_{\mathfrak{f}}$ be the number of roots of unity in $K(\mathfrak{f})$ and fix a complex place $v$ of $K(\mathfrak{f})$ above the complex place of $K$.

\medskip
\noindent
{\bf Stark's conjecture.} {\it There exists a unit $\mathbf{u}_{\rm Stark} \in K(\mathfrak{f})^\times$ such that
\begin{itemize}
\item if $w$ is a place of $K(\mathfrak{f})$ that does not divide $v$, then $| \mathbf{u}_{\rm Stark}|_w =1$,
\item for all $\rho \in \mathrm{Gal} (K(\mathfrak{f})/K)$, we have
\begin{equation} \label{E:StC}
\zeta_{\mathfrak{f}} ' ( \rho , 0) = -\frac{1}{e_\mathfrak{f}}  \log \left| \mathbf{u}_{\rm Stark}^{\rho} \right|_v ,  
\end{equation}
\item $K(\mathfrak{f})(\mathbf{u}_{\rm Stark}^{1/e_\mathfrak{f}} ) / K$ is an abelian extension. 
\end{itemize}
}

\medskip

The conjecture specifies the valuation of $\mathbf{u}_{\mathrm{Stark}}$ for every place of $K(\mathfrak{f})$ and so determines it up to a root of unity. One can get rid of this ambiguity using the ideal $\mathfrak{a}$: as explained in the Remark at the end of Section \ref{subsection:first_numerical_example}, if $\mathfrak{a}$ is coprime to $6 \mathfrak{f}$, then $\sigma_\mathfrak{a}  - \mathrm{N} (\mathfrak{a}) \in \Z [\mathrm{Gal}(K(\mathfrak{f})/K)]$ kills the roots of unity in $K(\mathfrak{f})$. It follows that $\mathbf{u}_{\rm Stark}^{\sigma_\mathfrak{a} - \mathrm{N} (\mathfrak{a})}$ is a well-defined unit in $K(\mathfrak{f})$, satisfying 
\begin{equation*} 
\zeta_{\mathfrak{f} , \mathfrak{a}} ' ( \mathfrak{b} , 0) = -\frac{1}{e_\mathfrak{f}}  \log \left| \mathbf{u}_{\rm Stark}^{(\sigma_\mathfrak{a} - \mathrm{N} (\mathfrak{a}) ) \sigma_\mathfrak{b} } \right|_v .
\end{equation*}

Theorem \ref{T:KLintro} relates our conjecture to that of Stark. We remark that his conjecture predicts that $\mathbf{u}_{\mathrm{Stark}}^{\mathrm{N}(\mathfrak{a})-\sigma_\mathfrak{a}}$ is an $e_\mathfrak{f}$-power in $K(\mathfrak{f})$ \cite[Section 4, Lemma 6]{Stark}. Assuming both conjectures hold we have
\begin{equation}
\mathbf{u}_{\rm Stark}^{\mathrm{N}(\mathfrak{a})-\sigma_{\mathfrak{a}}} = \mathbf{u}_{\mathfrak{f}, \mathfrak{a}}^{e_{\mathfrak{f}}}.
\end{equation}

Equation \eqref{E:StC} gives a formula for the modulus of special units in $K(\mathfrak{f})$ in terms of partial zeta functions. To quote Tate \cite[p. 95]{Tate}:
\begin{quote}
Trouver d'une telle façon des générateurs des corps de classes est la forme vague du 12ème problème de Hilbert, et la conjecture de Stark représente une importante contribution à ce problème. À vrai dire, c'est une contribution totalement inattendue~: Hilbert voulait qu'on trouve et discute des fonctions jouant [$\ldots$] le même rôle que la fonction exponentielle pour $\Q$ et que les fonctions modulaires elliptiques pour les corps quadratiques imaginaires.  Par contre, l'énoncé de Stark, utilisant directement les fonctions $L$, passe à côté de ces fonctions transcendantes inconnues attendues par Hilbert. Peut-être la connaissance de ces dernières sera-t-elle nécessaire pour démontrer la conjecture de Stark ?
\end{quote}

We believe that, for complex cubic fields, the elliptic gamma function is the ``fonction transcendante attendue par Hilbert''.

\section[Values of zeta functions]{Derivatives of partial zeta functions at $s=0$}

In this Section we prove Theorem \ref{T:KLintro}. Before that we briefly consider the case of  $\SL_2 (\Z),$ where the corresponding statement relies on the classical  Kronecker limit formula.
 It is not central for our purpose but it may help to understand our relevant case  $\SL_3 (\Z)$. 

\subsection{The case $\SL_2 (\Z)$} Let $\mathcal{L} \to \mathcal{H}$ be the Hodge line bundle on $\mathcal{H}$. Then $\mathcal{L}$ carries a canonical hermitian metric and we have an isomorphism of metrized bundles $\mathcal{L} \cong \mathcal{H} \times \C$,  where the latter is equipped with the metric $y^{-1} | \cdot |^2$ on the fiber over $\tau=x+iy \in \mathcal{H}$. In other words, in this case the basic geometric setup is the Hodge line bundle $(\mathcal{L},h) \to \mathcal{H}$, a hermitian holomorphic line bundle over the complex manifold $\mathcal{H}$. In this setting, Bismut-Gillet-Soul\'e \cite[Theorem 1.15]{BGS} have refined the canonical transgression form defined by Mathai and Quillen \cite{MathaiQuillen}  to a canonical $\partial \overline{\partial}$-transgression form $\psi$. Namely, there is a smooth function
\begin{equation*}
\psi \in A^0(\mathcal{L}),
\end{equation*}
rapidly decreasing along the fibers of $\mathcal{L}$, such that
\begin{equation*}
\partial \overline{\partial} \psi(t \mu ) = t \frac{d}{dt} \varphi(t \mu ),
\end{equation*}
where $\varphi$ is a constant multiple of the Mathai--Quillen Thom form.
We have
\begin{equation*}
\psi = e^{-\frac{|z|^2}{2y}},
\end{equation*}
and  one  is led to introduce, 
for any given $t>0$, the theta series 
\begin{equation*}
\theta_{\psi_t}(z,\tau) = \sum_{\omega \in   \Z + \Z \tau} e^{-t^2 \frac{|z+\omega|^2}{2y}}.
\end{equation*}
Its Mellin transform is the Eisenstein series 
\begin{equation}\label{eq1Epsi}
\begin{split}
E_\psi (z,\tau ; s) & = \int_0^\infty \theta_{\psi_t} (z,\tau) t^{2s} \frac{dt}{t} \\
& =\Gamma(s) 2^{s-1} y^s \sum_{\omega \in \Z + \Z \tau} \frac{1}{|z+ \omega|^{2s}},
\end{split}
\end{equation}
which possesses analytic continuation to $s\in \C$ and satisfies a functional equation, see \cite[Eqn. (32) p. 80]{Weil}.
Letting 
$$E_\psi = E_\psi (z, \tau  ;0)\in A^0 (\mathcal{L}-\{0\})^{\Z^2}$$ 
gives a way to make sense of the sum of the (non absolutely convergent) series 
$$
\frac{y}{2\pi} \sideset{}{'}\sum_{m,n \in \Z} \frac{e^{2i\pi (mv-nu)}}{|m\tau +n|^2},
$$
where $z= u\tau+v$. It follows from Kronecker's second limit formula \cite[Eqn. (39) p. 28]{Siegel} or \cite[\S 20.5]{LangElliptic}\footnote{A factor $(1-q_z)$ is missing in the final form in this reference.} that 
\begin{equation*} 
E_\psi  = - \log \left| \theta_0(z, \tau)e^{i\pi \tau B_2(u)}\right|^2, \quad \mbox{where} \quad B_2(u) = u^2-u+ \frac{1}{6}.
\end{equation*}
In this way the evaluation of $E_\psi$ at a CM-point is equal to $-\log | \cdot |$ of the evaluation of a holomorphic function.

\medskip

To phrase more  precisely the relation with partial zeta functions and the arithmetic of their special values in this case, let $F\subset \C$ be an imaginary quadratic field, and  let 
$ \mathfrak f$ and $\mathfrak a$ be proper coprime integral ideals of $F$. We assume that $\mathfrak a$ is a prime ideal of  norm $N=\mathrm{N}(\mathfrak a)$ and denote by $q$ the smallest positive integer in $\mathfrak{f} \cap \Z$ and by $w$ the number of roots of unity in $1+\mathfrak{f}$. Then there exists a quadratic irrational $\tau \in \mathcal{H} \cap F$ such that
\begin{equation}
\zeta'_{\mathfrak f,\mathfrak a}(\mathcal O_F, 0)= \zeta_{\mathfrak f} ' (\mathfrak a,0)-\mathrm{N}(\mathfrak a)\zeta_{\mathfrak f} ' (\mathcal O_F,0) = \frac{1}{w} \log \left| 
\frac{\theta_0( 1/q , \tau )^{N}}{\theta_0(N/q, N\tau)}\right|^{2}.  
\end{equation}
This identity implies Stark's conjecture for the imaginary quadratic field $F$: Complex Multiplication theory implies that 
if $\mathfrak f$ is not a prime power, then the complex number
$\theta_0( 1/q , \tau )^{N} / \theta_0(N/q, N\tau)$ is a unit in the ray class field $F(\mathfrak f).$
If $\mathfrak f=\mathfrak p^m,$ it is a unit outside $\mathfrak  p,$ see \cite[Prop. 2.4, p. 51]{deShalit}.

\subsection{The case $\SL_3 (\Z)$} In this case the object replacing the classical series \eqref{eq1Epsi} is an Eisenstein series $E_{\psi , \mathfrak{a}}$ that defines a $1$-form on the product of the symmetric space of $\SL_3 (\R)$ and $\mathbf{X}=\C\mathrm{P}^2-\R\mathrm{P}^2$. The strategy of proof for Theorem \ref{T:KLintro} can be summarised as
\begin{equation} \label{eq:main_strategy}
\begin{split}\frac{\sqrt{\pi}}{3} \zeta_{\mathfrak{f,a}} ' (\mathfrak{b} , 0) &= \int_{\langle \varepsilon \rangle} E_{\psi , \mathfrak{a}} (h , \mathfrak{fb}^{-1} , 0) \\
&= \int_{[b,  a]}  E_{\psi , \mathfrak{a}} (h, \mathfrak{fb}^{-1} , 0) \\
&=  \frac{\sqrt{\pi}}{3} \log \left| \mathbf{\Gamma}_{\mathfrak{a} , h} (\mathfrak{fb}^{-1}) \right|^2.
\end{split}
\end{equation}
That is, we will introduce two integration chains in the Satake compactification $\overline{\mathbf{S}}^\mathrm{Sat}$, denoted by $\langle \varepsilon \rangle$ and $[b,a]$ respectively. The chain $\langle \varepsilon \rangle$ corresponds to a closed geodesic in  $\Gamma \backslash (\mathbf{\overline{S}} \times \mathbf{X})$ for some arithmetic group $\Gamma \subset \mathrm{SL}_3(\Z)$, and we will show that the period of $E_{\psi,\mathfrak{a}}$ along this geodesic recovers $\zeta'_{\mathfrak{f},\mathfrak{a}}(\mathfrak{b},0)$. The other chain can be thought of as a generalised modular symbol. It is more precisely a limit of $1$-chains $[b, a]_{\leq t}$ as $t \to \infty$. The main step in the proof is the middle equality in \eqref{eq:main_strategy} showing that both periods of $E_{\psi,\mathfrak{a}}$ agree; the argument is geometric and is sketched in Figure \ref{fig:figure1}.

\begin{figure}
\begin{center}
\includegraphics[width=0.66\textwidth]{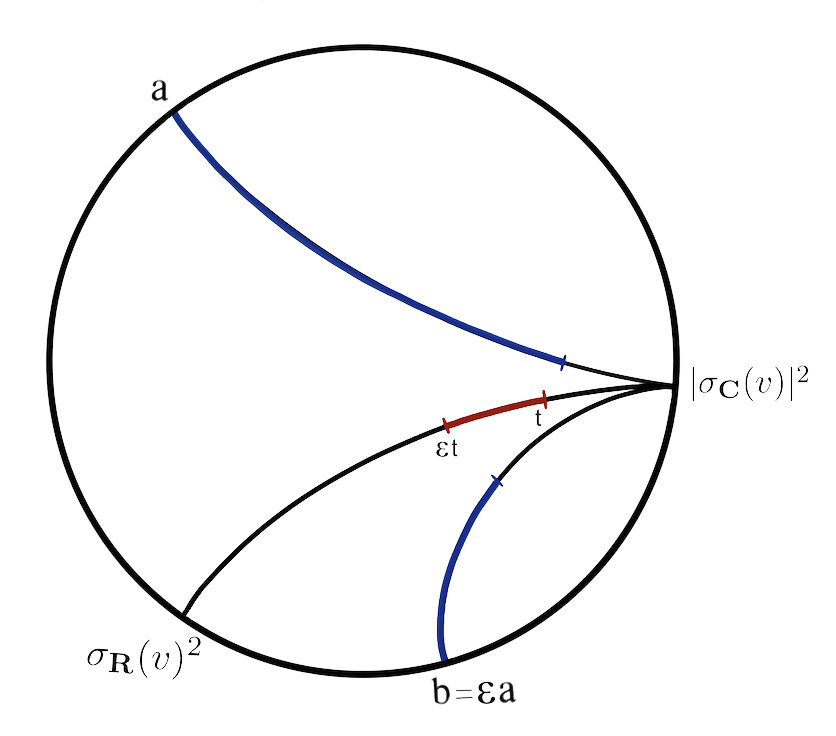}
\end{center}
\caption{Some geodesics in the Satake compactification  $\overline{\mathbf{S}}^{\mathrm{Sat}}$, with endpoints corresponding to degenerate quadratic forms on $K_\R$. The red segment is a fundamental domain for the action of $\varepsilon$ on the geodesic joining $\sigma_\R(v)^2$ to $|\sigma_\C(v)|^2$. The other two geodesics join $a(v)^2$ and $b(v)^2$ to $|\sigma_\C(v)|^2$; together they make up the generalised modular symbol $[b, a]$. The unit $\varepsilon$ maps one of the blue segments to the other. The integral of $E_{\psi, \mathfrak{a}}$ on the red segment equals $\sqrt{\pi} \zeta'_{\mathfrak{f},\mathfrak{a}}(\mathfrak{b},0)/3$. Thus it is independent of $t$, and letting $t \to \infty$ allows to integrate over $[b, a]_{\leq t}$ instead.}\label{fig:figure1}
\end{figure}

\medskip

We keep the notation from the previous sections; we fix a lattice $L=\mathfrak{fb}^{-1}$ and write $\Lambda = \mathrm{Hom}(L,\Z)$. For convenience we fix a basis of $L$ providing an isomorphism $L \simeq \Z^3$. 
For ${\bf K}=\Q, \R, \C$ we define $V_{\bf K}=\Lambda \otimes {\bf K}$. Recall also that we have fixed a complex embedding $\sigma_\C:K \hookrightarrow \C$ and an orientation of $K$ that is compatible with $\sigma_\C$ in the sense that the map $(\sigma_\R, \sigma_\C):K \otimes \R \to \R \times \C$ carries our orientation to the standard orientation on $\R \times \C$. Finally, we fix an ideal $\mathfrak{a}$ and an element $h \in L \otimes \Q$ satisfying the conditions of Section \ref{subsection:complex_cubic_fields_and_elliptic_gamma_function}; in particular, $h=\lambda/q$ with $\lambda$ admissible for $L$. We write $a$ for the element provided by Lemma \ref{L:L22} and set $b=\varepsilon a$.

\subsubsection{The Eisenstein series} 
Let
\begin{equation}
\X = (V_\C - \C \cdot V_\R)/\C^\times = \C \mathrm{P}^{2} - \R \mathrm{P}^{2}
\end{equation}
be the space of complex lines $l \subset V_\C$ such that $\overline{l} \neq l$. The group $G:=\mathrm{GL}(V_\R)^+ \simeq \mathrm{GL}_3(\R)^+$ acts naturally on $\X$. Denote by $Q$  the quadratic form on $V^\vee_\R$ defined by
\begin{equation} \label{eq:standard_quadratic_form}
Q(x) = x_1^2+ x_2^2 + x_3^2.
\end{equation}
Its isometry group $K=\SO(Q) \simeq \SO(3)$ is then a maximal compact subgroup of $G$. Let us denote by $Z^0 \simeq \R_{>0}$ the connected component of the center of $G$ and write
\begin{equation}
\bS = G/K \quad \text{and} \quad \overline{\bS} = \bS/Z^0.
\end{equation}

The map $s=gK \mapsto Q_s(\cdot) := Q(g^{-1} \cdot)$ gives a bijection between the points in $\bS$ and the positive definite quadratic forms on $V^\vee_\R$. In particular, the constant vector bundle
\begin{equation*}
\xymatrix{
\mathcal{V}^\vee := \bS \times \X \times V^\vee_\R \ar[d] \\
\bS \times \X
}
\end{equation*}
over $\bS \times \X$ with fiber $V^\vee_\R$ carries a tautological orientation and metric
$$
Q_{s}: \mathcal{V}^\vee_{s,l} \to \R_{\geq 0}, \quad (s,l) \in \bS \times \X.
$$
The group $G$ acts diagonally on $\bS \times \X$, and $\mathcal{V}^\vee$ has a $G$-equivariant structure defined by
$$
g \cdot (s,l,\mu) = (gs,gl,\mu(g^{-1} \cdot)).
$$
The metric $Q_s$ is $G$-equivariant, that is,
$$
Q_{gs}(g\mu) = Q_{s}(\mu), \qquad g \in G, \quad (s,l
) \in \bS \times \X, \quad \mu \in \mathcal{V}^\vee_{s,l}.
$$

Let us write $\mathcal{L}$ for the restriction of $\mathcal{O}(-1)$ to $\X$ and $\pi:\mathbf{S} \times \mathbf{X} \to \mathbf{X}$ for the projection. The line bundle $\pi^*\mathcal{L}^\vee$ is naturally a quotient of $\mathcal{V}^\vee$; we write $\mathcal{W} \subset \mathcal{V}^\vee$ for the kernel of quotient map. Using the metric $Q_s$, the bundle $\mathcal{V}^\vee$ decomposes as an orthogonal sum
\begin{equation} \label{eq:orthogonal_sum_bundle}
\mathcal{V}^\vee = \mathcal{W} \oplus \pi^* \mathcal{L}^\vee.
\end{equation}
Since $\mathcal{V}^\vee$ and $\pi^* \mathcal{L}^\vee$ are naturally oriented ($\mathcal{L}$ being a complex line bundle), this decomposition induces an orientation on $\mathcal{W}$. Let us denote by $e_\mathcal{W}$ the positive generator of the oriented real line bundle $\mathcal{W}$ with $|e_\mathcal{W}|_s=1$ for all $s$. Given a vector $v \in V^\vee_\R$ we obtain a function $(v,e_\mathcal{W}) \in \mathcal{C}^\infty (\mathbf{S} \times \mathbf{X})$: its value at $(s,l) \in \mathbf{S} \times \mathbf{X}$ is $(v,e_\mathcal{W}(l))_s$, where $(\cdot, \cdot)_s$ denotes the inner product associated with $s$ and $e_\mathcal{W}(l)$ is the positive generator of the line $\mathrm{Ann}(l) \subset V_\R^\vee$ satisfying $|e_\mathcal{W}(l)|_s=1$.

\begin{definition}
For $v \in V^\vee_\R$, let
$$
\psi(v) = -e^{-Q_s(v)} d(v,e_\mathcal{W}) \in \Omega^1(\mathbf{S} \times \mathbf{X}).
$$
As in the $\SL_2 (\R)$ case, the vector $v$ defines a section of $\mathcal{V}^\vee$, and $\psi(v)$ is a multiple of the pullback to $\mathbf{S} \times \mathbf{X}$ (via this section) of a double transgression of the canonical Thom form of $\mathcal{V}^\vee$ \cite{MathaiQuillen}.

For $t>0$, define
\begin{equation}
\Theta_{\psi_t}(h,L)=\sum_{\lambda \in h + L} \psi(t \lambda) \in \Omega^1(\mathbf{S} \times \mathbf{X})
\end{equation}
and set
\begin{equation}
E_\psi(h,L,s) = \int_0^\infty \Theta_{\psi_t}(h,L) t^{3s} \frac{dt}{t}.
\end{equation}
\end{definition}
The integral converges for any $s \in \C$; this follows from the standard argument splitting the integration domain into the intervals $(0,1)$ and $(1,\infty)$ and applying Poisson summation (note that $\psi$ is odd and so both $\psi$ and its Fourier transform vanish at the origin).

Note that for $g \in G$ we have
$$
g^* \psi(gv) = \psi(v).
$$
In particular, the differential forms $\Theta_\psi(h,L)$ and $E_\psi(h,L,s)$ are invariant under 
$$
\Gamma_1(h) = \{g \in \mathrm{GL}(L) \; | \; g h \equiv h \mod L \},
$$
and so we have
\begin{equation*}
\Theta_\psi(h,L) \in \Omega^1(\mathbf{S} \times \mathbf{X})^{\Gamma_1(h)} \simeq \Omega^1(\Gamma_1(h)\backslash (\mathbf{S} \times \mathbf{X}))
\end{equation*}
and
\begin{equation*}
E_\psi(h,L,0) \in \Omega^1(\overline{\mathbf{S}} \times \mathbf{X})^{\Gamma_1(h)} \simeq \Omega^1(\Gamma_1(h)\backslash (\overline{\mathbf{S}} \times \mathbf{X})).
\end{equation*}

\subsubsection{Generalized modular symbols}

A point $x=\langle \mathbf{x} \rangle \in \X$ determines a point $s_x \in \overline{\bS}^\mathrm{Sat}$ corresponding to the family of degenerate (rank two) quadratic forms obtained by rescaling the quadratic form $v \mapsto | \mathbf{x} (v) |^2$ for some non-zero generator $\mathbf{x}$ of $x$. 

The vector $a \in \Lambda$ spans a line $\langle a \rangle \subset V_\R$ that also determines a point $s_{\langle a \rangle}$ in $\overline{\bS}^\mathrm{Sat}$ corresponding to the degenerate (rank one) quadratic forms obtained by rescaling the quadratic form $v \mapsto a(v)^2$. 

The convex hull of $s_{\langle a \rangle}$ and $s_x$ is then a compact interval in $\overline{\bS}^\mathrm{Sat}$. We regard this interval as the image of a singular $1$-chain 
$$\Delta_{\langle a \rangle,x} : [0,1] \to \overline{\bS}^\mathrm{Sat}.$$ 
Its interior $\Delta_{\langle a \rangle,x}^0$ consists of the image in $\overline{\bS}$ of all quadratic forms of the form $v \mapsto t_1 a (v)^2 + t_2 |\mathbf{x} (v)|^2$. A choice of $\mathbf{x}$ gives a parametrisation
\begin{equation} \label{eq:param_gen_mod_symb_1}
\R_{>0} \xrightarrow{\sim} \Delta_{\langle a \rangle,x}^0, \quad t_1 \mapsto (v \mapsto t_1^{-2} a(v)^2 + t_1|\mathbf{x}(v) |^2) \ \ \mathrm{mod} \ Z^0.
\end{equation}
Let $x \in \X$. The \emph{generalised modular symbol} $
[a ,b]_x$ is the $1$-chain defined by
$$
[a,b]_x = \Delta_{\langle a \rangle,x}-\Delta_{\langle b \rangle,x}.
$$
The obvious properties of these chains justify our terminology (cf. \cite[Prop. 2.2]{AshRudolph}).

\subsubsection{The closed geodesic} Let us write $\langle \mathbf{x}\rangle \in \mathbf{X}$ for the complex line in $V_\C$ spanned by the complex embedding $\mathbf{x}:=\sigma_\C$. Since $h \in  \mathfrak{b}^{-1}$, we have $(\varepsilon -1)h \in \mathfrak{fb}^{-1}=L$ and so $\varepsilon \in \Gamma_1(h)$. For any $t_1 >0$, define a metric $H_{t_1} \in \mathbf{S}$ by
\begin{equation} \label{eq:closed_geodesic_metric}
H_{t_1}(v)=t_{1}^{-2} \sigma_\R (v)^2+ t_1|\sigma_{\C}(v)|^2, \quad v \in V^\vee_\R.
\end{equation}

Let us write $\overline{H}$ for the point in $\overline{\mathbf{S}}=\mathbf{S}/Z^0$ defined by a positive definite quadratic form $H \in \mathbf{S}$. Since $H_{\varepsilon_\R t_1}(\varepsilon v) = H_{t_1}(v)$, the assignment
$$
\iota: \R_{>0} \to \mathbf{S} \times \mathbf{X}, \qquad t_1 \mapsto (H_{t_1},\langle \mathbf{x} \rangle)
$$
induces a map
$$
\iota: \langle \varepsilon^n_\R \; | \; n \in \Z \rangle \backslash \R_{>0} \to \Gamma_1(h) \backslash (\overline{\mathbf{S}} \times \mathbf{X}).
$$
Its image is a closed oriented geodesic in $\Gamma_1(h) \backslash (\mathbf{S} \times \mathbf{X})$ that we denote by $\langle \varepsilon \rangle$.

\begin{lemma} \label{lemma:int_E_closed_geodesic}
$$
\int_{\langle \varepsilon \rangle} E_{\psi}(h,L,s) = \frac{1}{6}\Gamma(s) \Gamma\left( \frac{s+1}{2} \right) \sum_{v \in (1+L)/\mathcal{O}_\mathfrak{f}^{+,\times}} \frac{\mathrm{sgn}(\sigma_\R(v))}{|\mathrm{N}(v)|^s}.
$$
\end{lemma}
\begin{proof}
Writing $e_{t_1}$ for the positive generator of the fiber of $\mathcal{W}$ over $\iota(t_1)$, we have $(v,e_{t_1})=t_1^{-1} \sigma_\R(v)$ and hence 
$$
\iota^* \psi(v) = e^{- t_{1}^{-2} \sigma_\R(v)^2 - t_1|\sigma_\C(v)|^2} t_1^{-1} \sigma_\R(v) \frac{dt_1}{t_1}.
$$
This expression remains invariant upon replacing $(v,t_1)$ with $(\varepsilon v, \varepsilon_\R t_1)$ (use $\varepsilon_\R |\varepsilon_\C|^2 =1$). Since $h \equiv 1 \mod L$, this gives
\begin{equation*}
\begin{split}
\int_{\langle \varepsilon \rangle} & E_{\psi}(h,L,s) \\ &= \int_{t_0 \varepsilon_\R}^{t_0} \int_0^\infty \sum_{v \in 1 + L} e^{-t_1^{-2} t^2 \sigma_\R(v)^2 - t_1 t^2 |\sigma_\C(v)|^2} t t_1^{-1} \sigma_\R(v) \frac{dt_1}{t_1} t^{3s} \frac{dt}{t} \\
&= \sum_{v \in (1+L)/\mathcal{O}^{+,\times}_{\mathfrak{f}}} \sigma_\R(v) \int_0^\infty \int_0^\infty e^{-t_1^{-2} t^2 \sigma_\R(v)^2 - t_1 t^2 |\sigma_\C(v)|^2} t_1^{-1} t^{3s+1} \frac{dt_1 dt}{t_1 t}.
\end{split}
\end{equation*}
The integral can be computed using the change of variables $u=t_1^{-1} t$, $w=t_1 t^2$, which gives
\begin{equation*}
\begin{split}
\int_0^\infty \int_0^\infty & e^{-u^2 \sigma_\R(v)^2 - w|\sigma_\C(v)|^2} (u^2 w^{-1})^{1/3} (uw)^{s+1/3}  \frac{1}{3} \frac{du dw}{uw} \\
&= \frac{1}{3} \left(\int_0^\infty e^{- u^2 \sigma_\R(v)^2} u^{1+s} \frac{du}{u} \right) \left( \int_0^\infty e^{- w |\sigma_\C(v)|^2} w^{s} \frac{dw}{w} \right) \\
&= \frac{1}{|\sigma_\R(v)|^{1+s}|\sigma_\C(v)|^{2s}} \frac{1}{6}\Gamma\left(\frac{s+1}{2} \right) \Gamma\left(s\right),
\end{split}
\end{equation*}
proving the claim.
\end{proof}

For an auxiliary ideal $\mathfrak{a}$ as in Section \ref{subsection:complex_cubic_fields_and_elliptic_gamma_function}, define the ``$\mathfrak{a}$--smoothed'' forms
\begin{equation}
\begin{split}
\Theta_{\psi_t,\mathfrak{a}}(h,L) &= \Theta_{\psi_t}(h,\mathfrak{a}^{-1}L)-\mathrm{N}(\mathfrak{a}) \Theta_{\psi_t}(h,L) \\
E_{\psi,\mathfrak{a}}(h,L,s) &= E_\psi(h,\mathfrak{a}^{-1}L,s)-\mathrm{N}(\mathfrak{a}) E_{\psi}(h,L,s).
\end{split}
\end{equation}
Then Lemmas \ref{lemma:int_E_closed_geodesic} and \ref{lemma:derivative_zeta_sign} imply
\begin{equation} \label{eq:int_Eisenstein_derivative_zeta}
\int_{\langle \varepsilon \rangle} E_{\psi,\mathfrak{a}}(h,\mathfrak{fb}^{-1},0) = \frac{\sqrt{\pi}}{3} \zeta_{\mathfrak{f},\mathfrak{a}}'(\mathfrak{b},0).
\end{equation}

\subsubsection{Period on a generalized modular symbol} For $t_1>0$, define metrics $H_{a, t_1}, H_{\varepsilon_\R b, t_1} \in \mathbf{S}$ by
\begin{equation} \label{eq:modular_symbol_metric}
\begin{split}
H_{a, t_1}(v) &= t_1^{-2} a(v)^2 + t_1 |\sigma_\C(v)|^2, \\ 
H_{\varepsilon_\R b, t_1}(v) &= t_1^{-2} \varepsilon_\R^2 b(v)^2 + t_1 |\sigma_\C(v)|^2, \quad v \in V^\vee_\R.
\end{split}
\end{equation}
We obtain a map
$$
\iota_{\langle a \rangle}: \R_{>0} \to \overline{\mathbf{S}} \times \mathbf{X}, \quad t_1 \mapsto (\overline{H}_{a, t_1}, \langle \mathbf{x} \rangle).
$$
Contrary to the map $\iota$ in the previous subsection, this map does not have compact image in $\Gamma_1(h) \backslash (\overline{\mathbf{S}} \times \mathbf{X})$. However, as $t_1 \to \infty$ the image of $\iota_{\langle a \rangle}$ approaches the closed geodesic defined by $\iota$ (see Figure \ref{fig:figure1}). More precisely, we have the following.

\begin{lemma} \label{lemma:approx_closed_geodesic_modular_symbol}
Let $\alpha \in \Omega^1(\Gamma_1(h) \backslash (\overline{\mathbf{S}} \times \mathbf{X}))$. Then
$$
\int_{\langle \varepsilon \rangle} \alpha = \lim_{t_1 \to \infty} \int_{t_1 \varepsilon_\R}^{t_1} \iota_{\langle a \rangle}^* \alpha.
$$
\end{lemma}
\begin{proof}
For any $t_1 >0$, we have
$$
\int_{\langle \varepsilon \rangle} \alpha = \int_{t_1 \varepsilon_\R}^{t_1} \iota^* \alpha.
$$
Since $H_{x^3 a,x^2 t_1}=x^2 H_{a, t_1}$ for any $x>0$, we may, rescaling $a$ if necessary, assume that $H_{t_1}$ and $H_{a,t_1}$ have the same determinant. It then suffices to show that the distance $d(\overline{H}_{t_1},\overline{H}_{a,t_1})$ (where $d$ denotes a $G$--invariant Riemannian distance on $\overline{\mathbf{S}}$) is rapidly decreasing as $t_1 \to \infty$. Let $P$ be the real parabolic of $G/Z^0$ stabilizing the plane $\langle \mathrm{Re}(\mathbf{x}), \mathrm{Im}(\mathbf{x}) \rangle$.  Let $P=N_P A_P M_P$ be its Langlands decomposition and let $X_P = M_P / (M_P \cap K)$. Consider the associated horospherical decomposition (\cite[\textsection I.1]{BorelJi})
$$
N_P \times A_P \times X_P \simeq \overline{\mathbf{S}}, \quad (n, \exp(t_1 H), m (M_P \cap K)) \mapsto n \exp(t_1 H)mK.
$$
In the above coordinates the points $\overline{H}_{t_1}$ and $\overline{H}_{ a, t_1}$ of $\overline{\mathbf{S}}$ both have $A_P$-component $\exp(t_1 H)$ and differ only by multiplication by an element in $N_P$. The explicit formula for the distance in \cite[Prop. 4.3]{BorelStableRealCohomology} shows that $d(\overline{H}_{t_1},\overline{H}_{a,t_1})$ decreases exponentially as $t_1 \to \infty$.
\end{proof}

Consider the metric $H_{a,t_1}$ for some $t_1>0$ and the corresponding point $\iota_{\langle a \rangle}(t_1) \in \mathbf{S}$. Under this metric the subspaces $\ker(\sigma_\C)$ and $\ker(a)$ of $V_\R^\vee$ are orthogonal; the subspace $\ker(\sigma_\C)$ is the fiber over $\iota_{\langle a \rangle}(t_1)$ of the bundle $\mathcal{W}$, and $\sigma_\C$ identifies $\ker(a)$ with the fiber of $\pi^* \mathcal{L}^\vee$ (cf. \eqref{eq:orthogonal_sum_bundle}). Note that we must have $e_\mathcal{W}(a)<0$; otherwise, if $e_\mathcal{W}(a)>0$, then a basis $(\lambda,\mu)$ of $\ker(a)$ is oriented if and only if $(\sigma_\C(\lambda),\sigma_\C(\mu))$ is an oriented basis of $\C$, which contradicts the condition that $\mathbf{x} \in U_a$. It follows that $(v,e_\mathcal{W})=-t_1^{-1}a(v)$.

For positive real numbers $t_1, t$, let us define $\Theta_a(t_1,t,h,L)$ via
$$
\iota^*_{\langle a \rangle} \Theta_{\psi_t}(h,L) = \Theta_a(t_1,t,h,L) \frac{dt_1}{t_1}.
$$
We may then write
\begin{equation}
\Theta_a(t_1,t,h,L) = \sum_{v \in h+L} \psi^0(tv)_{a,t_1}
\end{equation}
with
\begin{equation} \label{eq:psi_0_def}
\psi^0(v)_{a,t_1} = -t_1^{-1} a(v) e^{-t_1^{-2} a(v)^2 - t_1|\sigma_\C(v)|^2}.
\end{equation}

We also have
$$
\iota^*_{\langle a \rangle} \Theta_{\psi_t,\mathfrak{a}}(h,L) = \Theta_{a,\mathfrak{a}}(t_1,t,h,L) \frac{dt_1}{t_1}
$$
with
$$
\Theta_{a,\mathfrak{a}}(t_1,t,h,L) := \Theta_a(t_1,t,h,\mathfrak{a}^{-1}L) - \mathrm{N}(\mathfrak{a})\Theta_a(t_1,t,h,L).
$$

\begin{lemma} \label{lemma:Eis_rapid_decrease_cusp}
For any $s \in \C$, the Mellin transform
$$
\int_0^\infty |\Theta_{a,\mathfrak{a}}(t_1,t,h,L)| t^s \frac{dt}{t}
$$
decreases rapidly as $t_1 \to 0$. In particular, the form
$
\iota_{\langle a \rangle}^* E_{\psi,\mathfrak{a}}(h,L,s) \in \Omega^1(\R_{>0})
$
is rapidly decreasing as $t_1 \to 0$.
\end{lemma}
\begin{proof}
It suffices to prove the statement for the integrals over the domains $t \in (1,\infty)$ and $t \in (0,1)$ separately. Consider first the domain $t>1$: we have
\begin{equation} \label{eq:Mellin_estimate_1_proof}
\int_1^\infty |\Theta_a(t_1,t,h,L)| t^s \frac{dt}{t} \leq \sum_{\lambda \in h+L} \int_1^\infty |\psi^0(t \lambda)_{a,t_1}| t^s \frac{dt}{t}
\end{equation}
and a similar inequality with $L$ replaced by $\mathfrak{a}^{-1}L$. Hence it suffices to show that the right hand side of \eqref{eq:Mellin_estimate_1_proof}, and the same sum with $L$ replaced by $\mathfrak{a}^{-1}L$, both decrease rapidly as $t_1 \to 0$. To see this, let us take $\lambda_1 \in L$ with $a(\lambda_1)=1$ (recall $a \in \mathrm{Hom}(L,\Z)$ is primitive). Then we may write $L= (L \cap H(a)) + \Z \lambda_1$ and, since $a(h)=0$, also $h+L=((h+L) \cap H(a))+\Z \lambda_1$. Also, the Schwartz function $\psi^0(v)_{a,t_1}$ is given by \eqref{eq:psi_0_def}; in particular, $\psi^0(v)_{a,t_1}=0$ when $v$ belongs to  $\ker(a)$. Hence we have
\begin{equation*}
\begin{split}
\sum_{\lambda \in h+L} |\psi^0(t \lambda )_{a,t_1}| &= \sum_{n \in \Z} \sum_{v \in (h+L) \cap H(a)} |\psi^0(t(n\lambda_1+v))| \\
&= \sum_{\substack{n \in \Z \\ n \neq 0}} t_1^{-1}t|n|e^{-t_1^{-2} t^2 n^2} \sum_{v \in (h+L) \cap H(a)} e^{-t_1 t^2|\sigma_\C(n \lambda_1+v)|^2}.
\end{split}
\end{equation*}
The inner sum is over a (translate of a) rank two lattice and so as $t_1 \to 0$ it is $O(t_1^{-1}t^{-2})$, where the implicit constant is independent of $n$. The rapid decrease of the right hand side of  \eqref{eq:Mellin_estimate_1_proof} for the lattices $L$ and $\mathfrak{a}^{-1}L$ follows.

Now consider the integration over $t \in (0,1)$. Here we will apply Poisson summation. Let us write $\varphi_{h,L,\mathfrak{a}}$ for the Schwartz form in $\mathcal{S}(V(\mathbf{A}_f)^\vee)$ defined by
\begin{equation} \label{eq:finite_Schwartz_smoothing}
\varphi_{h,L,\mathfrak{a}}(v) = 1_{\mathfrak{a}^{-1} L \otimes \hat{\Z}}(v-h)-\mathrm{N}(\mathfrak{a})1_{L \otimes \hat{\Z}}(v-h)
\end{equation}
(where $1_X$ denotes the indicator function of $X$),
so that
$$
\Theta_{a,\mathfrak{a}}(t_1,t,h,L) = \sum_{v \in V_\Q^\vee} \varphi_{h,L,\mathfrak{a}}(v) \psi^0(tv)_{a,t_1}.
$$
Writing $\hat{\varphi}_{h,L,\mathfrak{a}} \in \mathcal{S}(V(\mathbf{A}_f))$ and $\hat{\psi}^0_{a,t_1} \in \mathcal{S}(V_\R)$
for the Fourier transforms of $\varphi_{h,L,\mathfrak{a}}$ and $\psi^0_{a,t_1}$, Poisson summation gives
$$
\Theta_{a,\mathfrak{a}}(t_1,t,h,L) = t^{-3} \sum_{w \in V_\Q} \hat{\varphi}_{h,L,\mathfrak{a}}(w) \hat{\psi}^0_{a,t_1}(t^{-1}w).
$$
Note that the integral of $\varphi_{h,L,\mathfrak{a}}$ over the line $\mathbf{A}_f \gamma \subset V(\mathbf{A}_f)^\vee$ vanishes. This implies that the restriction of $\hat{\varphi}_{h,L,\mathfrak{a}}$ to $\ker(\gamma)$ is identically zero. The argument now proceeds as in the previous case; it suffices to observe that a direct computation of the Fourier transform $\hat{\psi}^0_{a,t_1}$ shows that, for a fixed $w_1 \in V_\Q$, $w_1 \not\in \ker(\gamma)$, one has a bound similar to the above, namely
$$
\sum_{w \in ker(\gamma)} \hat{\varphi}_{h,L,\mathfrak{a}}(w+n w_1) \hat{\psi}^0(t^{-1}(w+n w_1))_{a,t_1} = O(e^{-\alpha t_1^{-1}t^{-1}n^2}), \ n \in \Z, \ n \neq 0,
$$
for some $\alpha>0$.
\end{proof}

Using these lemmas we obtain the following explicit expression for $\zeta_{\mathfrak{f},\mathfrak{a}}'(\mathfrak{b},0)$.  Let us define
$$
F_{\mathfrak{f},\mathfrak{a}}(t_1,t,h,\mathfrak{b}) = \Theta_{\varepsilon_\R b,\mathfrak{a}}(t_1,t,h,L) - \Theta_{a,\mathfrak{a}}(t_1,t,h,L).
$$

\begin{lemma} \label{lemma:zeta_derivative_double_integral}
The function $F_{\mathfrak{f},\mathfrak{a}}(h,\mathfrak{b})$ belongs to $L^1(\R_{>0}^2, \tfrac{dt_1 dt}{t_1 t})$ and
\begin{equation*}
\frac{\sqrt{\pi}}{3} \zeta'_{\mathfrak{f},\mathfrak{a}}(\mathfrak{b},0) =  \int_0^\infty \int_0^\infty F_{\mathfrak{f},\mathfrak{a}}(t_1,t,h,\mathfrak{b}) \frac{dt_1}{t_1} \frac{dt}{t}.
\end{equation*}
\end{lemma}

The proof relies on the following bounds. Let us write $\hat{\psi}^0(\cdot)_{a,t_1}$ for the Fourier transform of $\psi^0_{a,t_1}$, defined by
$$
\hat{\psi}^0(w)_{a,t_1} = \int_{V_\R^\vee} \psi^0(v)_{a,t_1} e^{-2\pi i (v,w)} dv, \quad w \in V_\R,
$$
with respect to some invariant measure $dv$. Recall the metrics $H_{a,t_1}$, $H_{\varepsilon b,t_1}$ introduced in \eqref{eq:modular_symbol_metric}; part (i) of the next Lemma quantifies the statement that the distance in $\overline{\mathbf{S}}$ between the points defined by these metrics tends to zero as $t_1$ grows (cf. Figure \ref{fig:figure1}).

\begin{lemma} \label{lemma:arch_bounds_1}
\begin{enumerate}
\item[(i)] There exists $A>0$ such that
$$
\left| H_{a,t_1}(v) - H_{\varepsilon_\R b, t_1}(v) \right|  < A t_1^{-3/2} H_{a,t_1}(v)
$$
for every $v \in V^\vee_\R$ and $t_1 > 1$.

\item[(ii)] There exists $A>0$ such that
$$
|\psi^0(v)_{a,t_1} - \psi^0(v)_{\varepsilon_\R b, t_1}| < A t_1^{-3/2} e^{-H_{a,t_1}(v)/2}
$$
for every $v \in V^\vee_\R$ and $t_1 \gg 1$.

\item[(iii)] There exists $A>0$ such that
$$
|\hat{\psi}^0(w)_{a,t_1} - \hat{\psi}^0(w)_{\varepsilon_\R b, t_1}| < A t_1^{-3/2} e^{-H_{a,t_1}(w)/2}
$$
for every $w \in V_\R$ and $t_1 \gg 1$.
\end{enumerate}
\end{lemma}
\begin{proof}
Let us first prove (i). We have
\begin{equation} \label{eq:diff_metrics_1}
\left| H_{a,t_1}(v) - H_{\varepsilon_\R b, t_1}(v) \right| = t_1^{-2}|a(v)^2 - \varepsilon_\R^2 b(v)^2|.
\end{equation}
Let $C_0>0$ be such that 
$$
|a(v)| \leq C_0 (\sigma_\R(v)^2+|\sigma_\C(v)|^2)^{1/2}
$$
for all $v \in V^\vee_\R$.
Since $b(v)=(\varepsilon a)(v)=a(\varepsilon^{-1}v)$, we have
\begin{equation*}
\begin{split}
|a(v)-\varepsilon_\R b(v)| &= |a(v- \varepsilon_\R\varepsilon^{-1} v)| \\
& \leq C_0 |\sigma_\C(v- \varepsilon_\R \varepsilon^{-1} v)| \\
& \leq C_0 |\sigma_\C(v)| + C_0 \varepsilon_\R |\sigma_\C(\varepsilon^{-1}v)| \\
& = C_0(1+\varepsilon_\R |\varepsilon_\C|^{-1})|\sigma_\C(v)| \\
& \leq C_0 (1+\varepsilon_\R |\varepsilon_\C|^{-1}) t_1^{-1/2} H_{a,t_1}(v)^{1/2},
\end{split}
\end{equation*}
where the first inequality follows from the fact that $\sigma_\R(v-\varepsilon_\R \varepsilon^{-1} v)=0$. Similarly we have $a(v)^2 \leq t_1^2 H_{a,t_1}(v)$ and hence
\begin{equation*}
\begin{split}
|a(v)+\varepsilon_\R b(v)| &\leq 2|a(v)| + |a(v)-\varepsilon_\R b(v)| \\
& \leq (2 t_1 + C_0(1+\varepsilon_\R |\varepsilon_\C|^{-1})t_1^{-1/2}) H_{t_1}(v)^{1/2}.
\end{split}
\end{equation*}
The bound (i) follows from \eqref{eq:diff_metrics_1} by multiplying the above two inequalities.

Consider now (ii). Using the explicit expression \eqref{eq:psi_0_def}, we find
\begin{equation*}
\begin{split}
|\psi^0(v)_{a,t_1}-\psi^0(v)_{\varepsilon_\R b, t_1}|  \leq  \ & |e^{-H_{a,t_1}(v)} - e^{-H_{\varepsilon_\R b, t_1}(v)}| \cdot t_1^{-1}|\varepsilon_\R b(v)| 
\\ & + e^{-H_{a,t_1}(v)} \cdot t_1^{-1}|a(v)- \varepsilon_\R b(v)|,
\end{split}
\end{equation*}
and so (ii) follows from the inequalities above together with the elementary inequality $|e^x-1| \leq |x| e^{|x|}$ valid for $x \in \R$.

To prove (iii), consider the Schwartz form $\varphi$ on $V^\vee_\R$ defined by $\varphi(v)=v_1 e^{-\pi Q(v)}$ with $Q$ as in \eqref{eq:standard_quadratic_form}. Let $g_{a,t_1} \in G$ be 
such that $\pi^{1/2}(g_{a,t_1}v)_1=t_1^{-1}a(v)$, $\pi^{1/2}(g_{a,t_1}v)_2=t_1^{1/2} \mathrm{Re}(\sigma_\C(v))$ and $\pi^{1/2}(g_{a,t_1}v)_3=t_1^{1/2} \mathrm{Im}(\sigma_\C(v))$. 
Then
$$
\psi^0(v)_{a,t_1}=\varphi(g_{a,t_1}v)
$$
and hence 
$$
\hat{\psi}^0(w)_{a,t_1}=C \det(g_{a,t_1})^{-1} \varphi({^t}g_{a,t_1}^{-1}w)
$$
for some $C \neq 0$.
(Note that $\hat{\varphi}=C \varphi$ and $\det(g_{a,t_1})$ is independent of $t_1$.) Now (iii) can be proved in the same way as (ii).
\end{proof}

\begin{proof}[Proof of Lemma \ref{lemma:zeta_derivative_double_integral}]
For the integrability claim, note that the Mellin transform
$$
\int_0^\infty |F_{\mathfrak{f},\mathfrak{a}}(t_1,t,h,\mathfrak{b})| t^s \frac{dt}{t}
$$
converges absolutely for every $t_1>0$ and every $s \in \C$ by the standard argument using Poisson summation. By Lemma \ref{lemma:Eis_rapid_decrease_cusp} this function decreases rapidly as $t_1 \to 0$. Hence it suffices to prove that it is also rapidly decreasing as $t_1 \to \infty$. For this we split the integral into the domains $(0,1)$  and $(1,\infty)$. For $t>1$, the bound given by part (ii) of Lemma \ref{lemma:arch_bounds_1} shows that
\begin{equation} \label{lemma:arch_bounds_2}
\int_1^\infty |F_{\mathfrak{f},\mathfrak{a}}(t_1,t,h,\mathfrak{b})| t^s \frac{dt}{t} \leq C t_1^{-3/2} \int_1^\infty \sum_{v \in h + \mathfrak{a}^{-1}L} e^{-t^2 H_{a,t_1}(v)/2} t^s \frac{dt}{t}
\end{equation}
for some $C>0$, and so it suffices to show that the integral on the right hand side is bounded for $t_1>1$. To see this, let $\alpha \neq 0$ be such that the restriction of $a-\alpha \sigma_\R$ vanishes on the plane $\ker(\sigma_\C) \subset V_\R^\vee$ and consider the metric
$$
\tilde{H}_{t_1}(v) = t_1^{-2} \alpha^2 \sigma_\R(v)^2 + t_1 |\sigma_\C(v)|^2.
$$
It has the property $\tilde{H}_{t_1}(\varepsilon^{-1} v)=\tilde{H}_{\varepsilon_\R t_1}(v)$, hence the integral
$$
\int_1^\infty \sum_{v \in h + \mathfrak{a}^{-1}L} e^{-t^2 \tilde{H}_{t_1}(v)/2} t^s \frac{dt}{t}
$$
is also invariant upon replacing $t_1$ by $\varepsilon_\R t_1$, making it periodic and in particular bounded as a function of $t_1$. On the other hand, arguing as in the proof of Lemma \ref{lemma:arch_bounds_1} one shows that for some $A>0$
 we have
$$
|H_{a,t_1}(v) - \tilde{H}_{t_1}(v)| < A t_1^{-3/2} \tilde{H}_{t_1}(v)
$$
when $t_1>1$, and hence that $|e^{-H_{a,t_1}(v)}-e^{-\tilde{H}_{t_1}(v)}|< O(t_1^{-3/2}e^{-\tilde{H}_{t_1}(v)/2})$. These bounds imply that the integral on the right hand side of \eqref{lemma:arch_bounds_2} is bounded above by a constant independent of $t_1>1$. For $t<1$ one uses the same argument, having applied Poisson summation first and using part (iii) of Lemma \ref{lemma:arch_bounds_1} instead of (ii).

By \eqref{eq:int_Eisenstein_derivative_zeta} and Lemma \ref{lemma:approx_closed_geodesic_modular_symbol} we have
\begin{equation*}
\begin{split}
\frac{\sqrt{\pi}}{3} \zeta'_{\mathfrak{f},\mathfrak{a}}(\mathfrak{b},0) &= \lim_{T_1 \to \infty} \int_{T_1}^{T_1 \varepsilon_\R^{-1}} \iota_{\langle a \rangle}^* E_{\psi,\mathfrak{a}}(h,\mathfrak{fb}^{-1},0) \\
&= \lim_{T_1 \to \infty} \left[ \int_{T_0}^{T_1 \varepsilon_\R^{-1}} \iota_{\langle a \rangle}^* E_{\psi,\mathfrak{a}}(h,\mathfrak{fb}^{-1},0) - \int_{T_0}^{T_1} \iota_{\langle a \rangle}^* E_{\psi,\mathfrak{a}}(h,\mathfrak{fb}^{-1},0) \right]
\end{split}
\end{equation*}
for any $T_0>0$. By Lemma \ref{lemma:Eis_rapid_decrease_cusp}, one may let $T_0 \to 0$ and hence write
\begin{equation*}
\begin{split}
\frac{\sqrt{\pi}}{3} & \zeta'_{\mathfrak{f},\mathfrak{a}}(\mathfrak{b},0) \\ &= \lim_{T_1 \to \infty} \left[ \int_0^{T_1 \varepsilon_\R^{-1}} \iota_{\langle a \rangle}^* E_{\psi,\mathfrak{a}}(h,\mathfrak{fb}^{-1},0) - \int_0^{T_1} \iota_{\langle a \rangle}^* E_{\psi,\mathfrak{a}}(h,\mathfrak{fb}^{-1},0) \right] \\
&= \lim_{T_1 \to \infty} \int_0^{T_1}\int_0^\infty \left(\Theta_{a,\mathfrak{a}}(t_1 \varepsilon_\R^{-1},t,h,\mathfrak{fb}^{-1}) - \Theta_{a,\mathfrak{a}}(t_1,t,h,\mathfrak{fb}^{-1})\right) \frac{dt_1}{t_1} \frac{dt}{t}.
\end{split}
\end{equation*}
Using $\varepsilon_\R |\varepsilon_\C|^2=1$ and $b=\varepsilon a$, one checks that $\psi^0(\varepsilon^{-1} v)_{a,t_1 \varepsilon_\R^{-1}} = \psi^0(v)_{\varepsilon_\R b, t_1}$ and hence that
$$
\Theta_{a,\mathfrak{a}}(t_1 \varepsilon_\R^{-1}, t, h, \mathfrak{fb}^{-1}) = \Theta_{\varepsilon_\R b,\mathfrak{a}}(t_1,t,h, \mathfrak{fb}^{-1})
$$
since $\varepsilon$ stabilizes the coset $h+L$, proving the identity in the statement.
\end{proof}

We will prove Theorem \ref{T:KLintro} by showing that the double integral in the lemma above can be expressed in terms of the elliptic gamma function.

\subsubsection{Relation with the elliptic gamma function} Let us rewrite the integrand $F_{\mathfrak{f},\mathfrak{a}}$ in Lemma \ref{lemma:zeta_derivative_double_integral} as follows: define
\begin{multline*} 
f_{a,\mathfrak{a}}(t_1,t,v) = \sum_{j=0}^{\mathrm{N}(\mathfrak{a})-1} \left(  \psi^0 ( t (v + j\frac{\gamma}{\mathrm{N}(\mathfrak{a})}))_{\varepsilon_\R b,t_1} - \psi^0 (t (v + j \frac{\gamma}{\mathrm{N}(\mathfrak{a})}))_{a,t_1} \right) \\
- \mathrm{N}(\mathfrak{a}) \left(  \psi^0(tv)_{\varepsilon_\R b, t_1} - \psi^0(tv)_{a,t_1}  \right).
\end{multline*}
Then
\begin{equation*}
\begin{split}
F_{\mathfrak{f},\mathfrak{a}}(t_1,t,h,\mathfrak{b}) &= \sum_{v \in h + L} f_{a,\mathfrak{a}}(t_1,t,v) \\
&= \sum_{w \in L/\Z \gamma} \sum_{k \in \Z} f_{a,\mathfrak{a}}(t_1,t,h+w+k\gamma).
\end{split}
\end{equation*}

\begin{lemma} \label{lemma:exchange_integral_and_sum}
\begin{equation*}
\begin{split}
\int_0^\infty \int_0^\infty & F_{\mathfrak{f},\mathfrak{a}}(t_1,t,h,\mathfrak{b}) \frac{dt_1}{t_1} \frac{dt}{t} \\ &= \sum_{w \in L/\Z \gamma} \int_0^\infty \int_0^\infty \left( \sum_{k \in \Z } f_{a,\mathfrak{a}}(t_1,t,h+w+k\gamma) \right) \frac{dt_1}{t_1} \frac{dt}{t}.
\end{split}
\end{equation*}
\end{lemma}
\begin{proof}
Let
\begin{equation*} 
G_{\mathfrak{f},\mathfrak{a}}(t_1,t,h,\mathfrak{b}) = \sum_{w \in h + L}  \left| f_{a,\mathfrak{a}}(t_1,t,v) \right|.
\end{equation*}
The arguments in the proofs of Lemmas \ref{lemma:Eis_rapid_decrease_cusp} and \ref{lemma:zeta_derivative_double_integral} show that the integral $\smallint_1^\infty G_{\mathfrak{f},\mathfrak{a}}(t_1,t,h,\mathfrak{b}) \tfrac{dt}{t}$ is $O(t_1^{-3/2})$ as $t_1 \to \infty$ and is rapidly decreasing as $t_1 \to 0$. Hence
\begin{equation}
\label{lemma:exchange_sum_and_integral_1}
\int_0^\infty \int_1^\infty G_{\mathfrak{f},\mathfrak{a}}(t_1,t,h,\mathfrak{b}) \frac{dt}{t} \frac{dt_1}{t_1} < \infty,
\end{equation}
which implies that
\begin{equation*}
\begin{split}
\int_1^\infty \int_0^\infty & F_{\mathfrak{f},\mathfrak{a}}(t_1,t,h,\mathfrak{b}) \frac{dt_1}{t_1} \frac{dt}{t} \\ &= \sum_{w \in L/\Z \gamma} \int_1^\infty \int_0^\infty \left( \sum_{k \in \Z } f_{a,\mathfrak{a}}(t_1,t,h+w+k\gamma) \right) \frac{dt_1}{t_1} \frac{dt}{t}.
\end{split}
\end{equation*}
To prove the same identity with $\smallint_1^\infty \tfrac{dt}{t}$ replaced by $\smallint_0^1 \tfrac{dt}{t}$, let us first rewrite $F_{\mathfrak{f},\mathfrak{a}}$ using Poisson summation. To this end, choose $\alpha$ and $\beta$ as in Section \ref{subsection:elliptic_gamma_function} and let $\tau=\alpha(\mathbf{x})/\gamma(\mathbf{x})$ and $\sigma=\beta(\mathbf{x})/\gamma(\mathbf{x})$. First define a Schwartz function $\varphi$ on $\R^3$ by
$$
\varphi(x) = -\sqrt{\pi} x_1 e^{-\pi(x_1^2+ x_2^2+x_3^2)}, \quad x=(x_1,x_2,x_3) \in \R^3
$$
and let
$$
A_a = \pi^{-1/2}|\gamma(\mathbf{x})| \begin{pmatrix} \alpha(a)/|\gamma(\mathbf{x})| & 0 & 0 \\ \mathrm{Im}(\tau) & \mathrm{Im}(\sigma) & 0 \\ \mathrm{Re}(\tau) & \mathrm{Re}(\sigma) 
& 1 \end{pmatrix}, \quad m(t_1) = \begin{pmatrix} t_1^{-1} & 0 & 0 \\ 0 & t_1^{1/2} & 0 \\ 0 & 0 & t_1^{1/2} \end{pmatrix}.
$$
Note that $A_a$ is invertible and $m(t_1)$ has unit determinant. Then we have
$$
\psi^0(x_1 \alpha + x_2 \beta + x_3 \gamma)_{a,t_1} = \varphi(m(t_1)A_a x).
$$
Since $\varphi$ is (up to multiplication by constants) its own Fourier transform, this shows that the Fourier transform of $x \mapsto \psi^0(t(x_1 \alpha + x_2 \beta + x_3 \gamma))_{a,t_1}$ is
$$
x \mapsto \frac{1}{\det(A_a) t^3} \varphi(t^{-1}m(t_1^{-1}){^t}A_a^{-1} x).
$$
Similar considerations apply with $a$ replaced by $\varepsilon_\R b$ (note that, since the metrics $H_{\varepsilon_\R b,t_1}(\varepsilon v)=H_{a,t_1 \varepsilon_\R^{-1}}(v)$, we have $\det(A_a)=\det(A_{\varepsilon_\R b})$). Hence writing $(\alpha_1,\alpha_2,\alpha_3)=(\alpha,\beta,\gamma)$ and setting 
$$
\varphi_{\mathfrak{a},h}(m_1,m_2,m_3) = \varphi_{h,L,\mathfrak{a}}(\Sigma m_i \alpha_i) \in \mathcal{S}(\mathbf{A}_f^3)
$$
with $\varphi_{h,L,\mathfrak{a}}$ defined by \eqref{eq:finite_Schwartz_smoothing}, Poisson summation reads
\begin{equation*}
\begin{split}
& - F_{\mathfrak{f},\mathfrak{a}}(t_1,t,h,\mathfrak{b}) \\ &= \sum_{m \in \Q^3} \varphi_{\mathfrak{a},h}(m)  \left(\psi^0\left(\sum t m_i \alpha_i \right)_{a,t_1} -\psi^0\left( \sum t m_i \alpha_i \right)_{\varepsilon_\R b,t_1} \right)\\
&= \frac{\det(A_a)^{-1}}{C  t^3} \sum_{m \in \Q^3} \hat{\varphi}_{\mathfrak{a},h} (m) \left(\varphi(t^{-1} m(t_1^{-1}){^t}A_a^{-1} m) - \varphi(t^{-1} m(t_1^{-1}) {^t}A_{\varepsilon_\R b}^{-1} m)   \right)
\end{split}
\end{equation*}
for some constant $C \neq 0$. Now the Schwartz function $\varphi_{h,L,\mathfrak{a}}$ is a linear combination of functions of the form $\varphi_{\langle \alpha, \beta \rangle} \otimes \varphi_{\langle \gamma \rangle}$, where $\varphi_{\langle \alpha, \beta \rangle} \in \mathcal{S}(\langle \alpha, \beta \rangle \otimes_\Q \mathbf{A}_f)$ and $\varphi_{\langle \gamma \rangle} \in \mathcal{S}(\Q \gamma \otimes_\Q \mathbf{A}_f)$, and where the Fourier transform of $\varphi_{\langle \gamma \rangle}$ satisfies $\hat{\varphi}_{\langle \gamma \rangle}(0)=0$. Using the bounds in Lemma \ref{lemma:arch_bounds_1} we obtain
\begin{equation*}
\begin{split}
\frac{1}{C \det(A_a) t^3} \sum_{m \in \Q^3} & |\hat{\varphi}_{\mathfrak{a},h} (m) \left(\varphi(t^{-1} m(t_1^{-1}){^t}A_a^{-1} m) - \varphi(t^{-1} m(t_1^{-1}) {^t}A_{\varepsilon_\R b}^{-1} m)   \right)| \\
& \leq C_1 \det(A_a^{-1}) t^{-3} t_1^{-3/2} \sum_{\substack{m \in \Z^3 \\ \mathrm{N}(\mathfrak{a}) \nmid m_3}} \varphi^0(C_2 t^{-1} m(t_1^{-1}) {^t}A_a^{-1} m),
\end{split}
\end{equation*}
where $C_1, C_2 >0$ and $\varphi^0(x)=e^{-x_1^2-x_2^2-x_3^2}$ is the standard Gaussian. Since the sum contains no terms with $m_3 = 0$, as in the proof of Lemma \ref{lemma:Eis_rapid_decrease_cusp} one shows that the integral $\smallint_0^1 \tfrac{dt}{t}$ of this expression is rapidly decreasing as $t_1 \to 0$. Applying Poisson summation again to this last expression shows that it equals
$$
C' t_1^{-3/2} \sum_{m_i \in \Q^3} \varphi_f(m) \varphi^0(C_2^{-1} t m(t_1) A_a m).
$$
for an appropriate Schwartz function $\varphi_f \in \mathcal{S}(\mathbf{A}_f^3)$. The argument in the proof of Lemma \ref{lemma:zeta_derivative_double_integral} shows that the integral $\smallint_0^1 \tfrac{dt}{t}$ of the sum in this expression is bounded for $t_1>1$, and so the same integral of the whole expression is $O(t_1^{-3/2})$ as $t_1 \to \infty$. Overall, these estimates imply that we can exchange sum and integral and so write 
\begin{equation*}
\begin{split}
& -C \det(A_a)  \int_0^1 \int_0^\infty  F_{\mathfrak{f},\mathfrak{a}}(t_1,t,h,\mathfrak{b}) \frac{dt_1}{t_1} \frac{dt}{t} \\
&= \sum_{m_1,m_2 \in \Q} \int_0^1 \int_0^\infty \sum_{m_3 \in \Q} \hat{\varphi}_{\mathfrak{a},h}(m) \\
& \qquad \qquad \qquad \qquad \times \left(\varphi(t^{-1} m(t_1^{-1}){^t}A_a^{-1} m) - \varphi(t^{-1} m(t_1^{-1}) {^t}A_{\varepsilon_\R b}^{-1} m)   \right) \frac{dt_1}{t_1} \frac{dt}{t^4}.
\end{split}
\end{equation*}
To reach the expression in the statement we apply Poisson summation again, this time to the right hand side of this equation. That is, define a function $g:\R^2 \to \C$ by
\begin{equation*}
\begin{split}
g(x_1,x_2) = &\frac{1}{C \det(A_a)}   \int_0^1 \int_0^\infty \sum_{m_3 \in \Q} \hat{\varphi}_{\langle \gamma \rangle }(m_3 \gamma) \\
&\times\left( \varphi(t^{-1} m(t_1^{-1}) \cdot {^t}A_a^{-1} \cdot {^t}(x_1,x_2,m_3)) \right. \\ & \left. \qquad  - \varphi(t^{-1} m(t_1^{-1}) \cdot {^t}A_{\varepsilon_\R b}^{-1} \cdot {^t}(x_1,x_2,m_3))   \right) \frac{dt_1}{t_1} \frac{dt}{t^4}.
\end{split}
\end{equation*}
Our goal is to show that Poisson summation applies to $\Sigma_{m_1,m_2 \in \Q} \hat{\varphi}_{\langle \alpha, \beta \rangle}(m_1 \alpha + m_2 \beta) g(m_1,m_2)$. To see this, note that the estimates above show that $g$ is a continuous and integrable function on $\R^2$ and that for any $l \in \Q^2$ and lattice $L_1$ in $\Q^2$, the series $G(x)=\Sigma_{m \in l+L_1} g(x+m)$ converges absolutely to a continuous function $G$ on $\R^2/L_1$. The Fourier coefficients of $G$ are $G(\lambda) = \hat{g}(\lambda)e^{2\pi i \lambda(l)}$, where $\lambda \in L_1^\vee=\mathrm{Hom}(L_1,\Z)$ and $\hat{g}$ denotes the Fourier transform of $g$. Since $G$ is continuous, Poisson summation (which is equivalent to the pointwise convergence of the Fourier series of $G$ at $x=0$) holds if we can show that $\Sigma_{\lambda \in L_1^\vee}|\hat{g}(\lambda)| < \infty$. 
This last convergence follows from \eqref{lemma:exchange_sum_and_integral_1} and the explicit computation below of the right hand side of the Lemma's statement (cf. sentence after \eqref{eq:def_I_-+}).
\end{proof}

We can now explain how this lemma implies Theorem \ref{T:KLintro}. By Lemma \ref{lemma:zeta_derivative_double_integral}, it suffices to compute the integrals on the right hand side; for this it will be convenient to change variables and set 
\begin{equation} \label{eq:change_of_vars}
u = \pi |\gamma(\mathbf{x})|^{-2} t_1^{-1}t^{-2}, \quad v= t_1^{-1} t.
\end{equation}
Then $(t_1 t)^{-1} dt_1 dt = 3^{-1} (uv)^{-1} dv du$ and so for fixed $w \in L/\Z \gamma$ we have
\begin{equation*}
\begin{split}
\int_0^\infty \int_0^\infty & \left( \sum_{k \in \Z} f_{a,\mathfrak{a}}(t_1,t,h+w+k \gamma) \right) \frac{dt_1}{t_1} \frac{dt}{t} \\ &= \frac{1}{3} \int_0^\infty \int_0^\infty \left( \sum_{k \in \Z } f^\sharp_{a,\mathfrak{a}}(u,v,h+w +k\gamma) \right) \frac{dv}{v} \frac{du}{u},
\end{split}
\end{equation*}
where we write
\begin{multline*} 
f^\sharp_{a,\mathfrak{a}}(u,v,w) = \sum_{j=0}^{\mathrm{N}(\mathfrak{a})-1} \left(  \psi^\sharp(u,v, w + j \frac{\gamma}{\mathrm{N}(\mathfrak{a})})_{\varepsilon_\R b,t_1} - \psi^\sharp(u,v, w + j \frac{\gamma}{\mathrm{N}(\mathfrak{a})})_{a,t_1} \right) \\
- \mathrm{N}(\mathfrak{a}) \left( \psi^\sharp(u,v,w)_{\varepsilon_\R b, t_1} - \psi^\sharp(u,v,w)_{a,t_1}  \right)
\end{multline*}
with
\begin{equation*}
\begin{split}
\psi^\sharp(u,v,w)_{a,t_1} & = -v a(w) e^{-v^2 a(w)^2} e^{-\pi u^{-1} |w(\mathbf{x})/\gamma(\mathbf{x})|^2} \\
\psi^\sharp(u,v,w)_{\varepsilon_\R b,t_1} & = -v \varepsilon_\R b(w) e^{-v^2 \varepsilon_\R^2 b(w)^2} e^{-\pi u^{-1} |w(\mathbf{x})/\gamma(\mathbf{x})|^2}.
\end{split}
\end{equation*}
By Poisson summation we have
\begin{equation*} 
\begin{split}
\sum_{k \in \Z} e^{-\pi u^{-1}|k+w(\mathbf{x})/\gamma(\mathbf{x})|^2} &=
 e^{-\pi u^{-1} \mathrm{Im}(w(\mathbf{x})/\gamma(\mathbf{x}))^2} \sum_{k \in \Z} e^{-\pi u^{-1} (k+\mathrm{Re}(w(\mathbf{x})/\gamma(\mathbf{x})))^2} \\
 &= u^{1/2} e^{-\pi u^{-1} \mathrm{Im}(w(\mathbf{x})/\gamma(\mathbf{x}))^2} \sum_{k \in \Z} e^{2\pi i k \mathrm{Re}(w(\mathbf{x})/\gamma(\mathbf{x}))} e^{-\pi k^2 u}.
\end{split}
\end{equation*}
Since $h$ is a rational multiple of $\gamma$ and hence $h(\mathbf{x})/\gamma(\mathbf{x}) \in \R$, this gives
\begin{multline*} 
\sum_{k \in \Z} f^\sharp_{a,\mathfrak{a}}(u,v,h+w+k \gamma) = -\mathrm{N}(\mathfrak{a}) (va(w) e^{-v^2 a(w)^2} - v \varepsilon_\R  b(w) e^{-v^2 \varepsilon_\R^2 b (w)^2} ) \\ \times u^{1/2} \sum_{\substack{k \in \Z \\ \mathrm{N}(\mathfrak{a}) \nmid k}} e^{2\pi i k \mathrm{Re}((h(\mathbf{x})+w(\mathbf{x}))/\gamma(\mathbf{x}))} e^{-\pi (k^2 u + \mathrm{Im}(w(\mathbf{x})/\gamma(\mathbf{x}))^2 u^{-1})}. 
\end{multline*}
Since
$$
\int_0^\infty y e^{-v^2 y^2} dv = \mathrm{sgn}(y) \int_0^\infty e^{-x^2} dx = \frac{\sqrt{\pi}}{2} \mathrm{sgn}(y),
$$
many of the terms in the sum in Lemma \ref{lemma:exchange_integral_and_sum} cancel: the only non-zero terms are those with $
\mathrm{sgn}(a(w)) \neq \mathrm{sgn}(b(w)).
$
Let us write
\begin{equation}
\begin{split}
\overline{C}_{\texttt{+-}}=\overline{C_{\texttt{+-}}(a,b)}=\{\delta \in L \, | \, \delta(a) \geq 0, \ \delta(b) \leq 0 \} \\
\overline{C}_{\texttt{-+}}=\overline{C_{\texttt{-+}}(a,b)}=\{\delta \in L \, | \, \delta(a) \leq 0, \ \delta(b) \geq 0 \} \\
\end{split}
\end{equation}
for the closures of the cones $C_{\texttt{+-}}(a,b)$ and $C_{\texttt{-+}}(a,b)$, and define the function
$$
\chi(w)=(\mathrm{sgn}(a(w))-\mathrm{sgn}(b(w)))/2
$$
(with the convention that $\mathrm{sgn}(0)=0$), which has support in $\overline{C}_{\texttt{+-}}\cup \overline{C}_{\texttt{-+}}$. Then the sum in Lemma \ref{lemma:exchange_integral_and_sum} may be written as $I_{\texttt{+-}}(h)+I_{\texttt{-+}}(h)$, where
\begin{multline} \label{eq:def_I_+-}
I_{\texttt{+-}}(h) = -\frac{\sqrt{\pi}}{3} \mathrm{N}(\mathfrak{a})  \sum_{w \in \overline{C}_{\texttt{+-}}/\Z \gamma} \chi(w) \sum_{\substack{k \in \Z \\ \mathrm{N}(\mathfrak{a}) \nmid k}}  e^{2\pi i k \mathrm{Re}((h(\mathbf{x})+w(\mathbf{x}))/\gamma(\mathbf{x}))} \\  \times \int_0^\infty e^{-\pi (k^2 u + \mathrm{Im}(w(\mathbf{x})/\gamma(\mathbf{x}))^2 u^{-1})} \frac{du}{u^{1/2}},
\end{multline}
\begin{multline} \label{eq:def_I_-+}
I_{\texttt{-+}}(h) = -\frac{\sqrt{\pi}}{3} \mathrm{N}(\mathfrak{a})  \sum_{w \in  \overline{C}_{\texttt{-+}}/\Z \gamma} \chi(w)  \sum_{\substack{k \in \Z \\ \mathrm{N}(\mathfrak{a}) \nmid k}}  e^{2\pi i k \mathrm{Re}((h(\mathbf{x})+w(\mathbf{x}))/\gamma(\mathbf{x}))} \\  \times \int_0^\infty e^{-\pi (k^2 u + \mathrm{Im}(w(\mathbf{x})/\gamma(\mathbf{x}))^2 u^{-1})} \frac{du}{u^{1/2}}.
\end{multline}
Note that the above infinite sums converge absolutely since the condition $\mathbf{x} \in U_a \cap U_b$ guarantees the function $w \mapsto |\mathrm{Im}(w(\mathbf{x})/\gamma(\mathbf{x}))|$ grows at least linearly on the cones $\overline{C}_{\texttt{+-}}$ and $\overline{C}_{\texttt{-+}}$.

Let us write $e(z)=e^{2\pi i z}$. Since
$$
\int_0^\infty e^{-\pi (k^2 u + y^2 u^{-1})} \frac{du}{u^{1/2}} = K_{1/2}(\pi^{1/2}|k|,\pi^{1/2}|y|) = \frac{1}{|k|}e^{-2\pi |ky|},
$$
(cf. \cite[Ch. 20, \textsection 3]{LangElliptic}), we have
\begin{equation*}
\begin{split}
& e^{2\pi i k \mathrm{Re}((h(\mathbf{x})+w(\mathbf{x}))/\gamma(\mathbf{x}))} \int_0^\infty e^{-\pi (k^2 u + \mathrm{Im}(w(\mathbf{x})/\gamma(\mathbf{x}))^2 u^{-1})} \frac{du}{u^{1/2}} \\
& \qquad = \left\{ \begin{array}{cc}\frac{1}{|k|}e\left(k \frac{h(\mathbf{x})+w(\mathbf{x})}{\gamma(\mathbf{x})}\right), & \text{ if } k\, \mathrm{Im}(w(\mathbf{x})/\gamma(\mathbf{x}))>0, \\ \frac{1}{|k|}e(k\overline{(h(\mathbf{x})+w(\mathbf{x}))/\gamma(\mathbf{x})}), & \text{ if } k\, \mathrm{Im}(w(\mathbf{x})/\gamma(\mathbf{x}))<0.  \end{array}  \right.
\end{split}
\end{equation*}
Since $\mathrm{Im}(w(\mathbf{x})/\gamma(\mathbf{x}))$ is negative for $w \in C_{\texttt{+-}}(a,b)$ and positive for $w \in C_{\texttt{-+}}(a,b)$, substituting in \eqref{eq:def_I_+-} and \eqref{eq:def_I_-+} we obtain
\begin{equation*}
\begin{split}
I_{\texttt{+-}}(h)=\frac{\sqrt{\pi}}{3} \sum_{w \in \overline{C}_{\texttt{+-}}/\Z \gamma} \chi(w) \sum_{\substack{k \geq 1 \\ \mathrm{N}(\mathfrak{a}) \nmid k}} -  \frac{\mathrm{N}(\mathfrak{a})}{k}( & e\left(k \frac{\overline{h(\mathbf{x})+w(\mathbf{x})}}{\overline{\gamma(\mathbf{x})}} \right) \\ & + e\left(-k \frac{h(\mathbf{x})+w(\mathbf{x})}{\gamma(\mathbf{x})} \right))
\end{split}
\end{equation*}
and
\begin{equation*}
\begin{split}
I_{\texttt{-+}}(h)=\frac{\sqrt{\pi}}{3} \sum_{w \in \overline{C}_{\texttt{-+}}/\Z \gamma} \chi(w) \sum_{\substack{k \geq 1 \\ \mathrm{N}(\mathfrak{a}) \nmid k}}  - \frac{\mathrm{N}(\mathfrak{a})}{k} ( & e\left(-k \frac{\overline{h(\mathbf{x})+w(\mathbf{x})}}{\overline{\gamma(\mathbf{x})}} \right) \\ &  + e\left(k \frac{h(\mathbf{x})+w(\mathbf{x})}{\gamma(\mathbf{x})} \right)) .
\end{split}
\end{equation*}
Since $\overline{e(z)}=e(-\overline{z})$, the identity $\Sigma_{k \geq 1, N \nmid k} f(k)=\Sigma_{k \geq 1}(f(k)-f(Nk))$ applied to the inner sum shows that (we abbreviate $N=\mathrm{N}(\mathfrak{a})$)
\begin{equation*} 
I_{\texttt{+-}}(h) = \frac{\sqrt{\pi}}{3} \sum_{w \in \overline{C}_{\texttt{+-}}/\Z \gamma} \chi(w) \log\left|\frac{(1-e(-(h(\mathbf{x})+w(\mathbf{x}))/\gamma(\mathbf{x})))^N}{1-e(-
N(h(\mathbf{x})+w(\mathbf{x}))/\gamma(\mathbf{x}))} \right|^2
\end{equation*}
and
\begin{equation*} 
I_{\texttt{-+}}(h) = \frac{\sqrt{\pi}}{3} \sum_{w \in \overline{C}_{\texttt{-+}}/\Z \gamma} \chi(w) \log\left|\frac{(1-e((h(\mathbf{x})+w(\mathbf{x}))/\gamma(\mathbf{x})))^N}{1-e(
N(h(\mathbf{x})+w(\mathbf{x}))/\gamma(\mathbf{x}))} \right|^2.
\end{equation*}

Replacing $w$ by $-w$ and using that $\chi$ is odd we may rewrite this as
\begin{equation} \label{eq:int_I_+-_3}
I_{\texttt{+-}}(h) = -\frac{\sqrt{\pi}}{3} \sum_{w \in \overline{C}_{\texttt{-+}}/\Z \gamma} \chi(w) \log\left|\frac{(1-e((-h(\mathbf{x})+w(\mathbf{x}))/\gamma(\mathbf{x})))^N}{1-e(
(-h(\mathbf{x})+w(\mathbf{x}))/(\gamma(\mathbf{x})/N))} \right|^2
\end{equation}
and
\begin{equation} \label{eq:int_I_-+_3}
I_{\texttt{-+}}(h) = -\frac{\sqrt{\pi}}{3} \sum_{w \in \overline{C}_{\texttt{+-}}/\Z \gamma} \chi(w) \log\left|\frac{(1-e(-(-h(\mathbf{x})+w(\mathbf{x}))/\gamma(\mathbf{x})))^N}{1-e(
-(-h(\mathbf{x})+w(\mathbf{x}))/(\gamma(\mathbf{x})/N))} \right|^2.
\end{equation}

Let us write $1_X$ for the indicator function of a set $X$. Then 
\begin{equation}
\begin{split}
\chi|_{\overline{C}_{\texttt{+-}}}  &= 1_{ C_{\texttt{+-}}(a,b)} - \frac{1}{2} \cdot 1_{\overline{C}_{\texttt{+-}}\cap H(b)} + \frac{1}{2} \cdot 1_{\overline{C}_{\texttt{+-}} \cap H(a)} \\
\chi|_{\overline{C}_{\texttt{-+}}}  &= -1_{ C_{\texttt{-+}}(a,b)} + \frac{1}{2} \cdot 1_{\overline{C}_{\texttt{-+}}\cap H(a)} - \frac{1}{2} \cdot 1_{\overline{C}_{\texttt{-+}} \cap H(b)},
\end{split}
\end{equation}
and so adding \eqref{eq:int_I_+-_3} and \eqref{eq:int_I_-+_3} gives
\begin{equation} \label{eq:log_gamma_plus_bdy_term}
\begin{split}
\frac{3}{\sqrt{\pi}} (I_{\texttt{+-}}(h)+I_{\texttt{-+}}(h)) &= \log \left|\frac{\Gamma_{a,b}(h(\mathbf{x}),\mathbf{x};\mathfrak{a}^{-1}L)}{\Gamma_{a,b}(h(\mathbf{x}),\mathbf{x};L)^N}
\right|^2 + \text{ boundary term},
\end{split}
\end{equation}
where the boundary term arises from vectors $w$ lying on the boundary of the cone $\overline{C}_{\texttt{+-}} \cup \overline{C}_{\texttt{-+}}$, i.e. on $H(a)$ or $H(b)$. It remains to show that this boundary term vanishes. To see this, let us pick $\alpha \in H(b)$ such that $\alpha(a)>0$ and $L \cap H(b)=\Z \alpha + \Z \gamma$. Then $(\gamma,\alpha)$ is an oriented basis of $H(b)$ and so
$
\tau := \alpha(\mathbf{x})/\gamma(\mathbf{x})
$
has negative imaginary part. Let us write $q_z=e^{2 \pi i z}$ and $z_0= h(\mathbf{x})/\gamma(\mathbf{x})=n/q$ (where $n \in \Z$ is such that $\lambda=n\gamma$) and define
\begin{equation} \label{eq:def_theta_0}
\theta_0(z,-\tau) = \prod_{n \geq 0} (1-q_z q_{-\tau}^{n})(1-q_z^{-1}q_{-\tau}^{n+1}).
\end{equation}
The contributions to the boundary term coming from the sums over $\overline{C}_{\texttt{-+}} \cap H(b)$ and $\overline{C}_{\texttt{+-}} \cap H(b)$ are respectively
\begin{equation*}
\begin{split}
-\frac{1}{2} \sum_{\substack{w \in (\overline{C}_{\texttt{-+}} \cap H(b))/\Z \gamma \\ w \neq 0}} & \log\left|\frac{(1-e((-h(\mathbf{x})+w(\mathbf{x}))/\gamma(\mathbf{x})))^N}{1-e(
(-h(\mathbf{x})+w(\mathbf{x}))/(\gamma(\mathbf{x})/N))} \right|^2 
\\ & = - \sum_{n \geq 1} \log \left|\frac{(1-q_{z_0}^{-1}q_{-\tau}^{n})^N}{1-q^{-1}_{Nz_0}q^{n}_{-N\tau}} \right|
\end{split}
\end{equation*}
(we have omitted $w =0$ because it cancels with the corresponding term for $H(a)$) and
\begin{equation*}
\begin{split}
-\frac{1}{2} \sum_{w \in (\overline{C}_{\texttt{+-}} \cap H(b))/\Z \gamma} & \log\left|\frac{(1-e(-(-h(\mathbf{x})+w(\mathbf{x}))/\gamma(\mathbf{x})))^N}{1-e(
-(-h(\mathbf{x})+w(\mathbf{x}))/(\gamma(\mathbf{x})/N))} \right|^2 
\\ & = -\sum_{n \geq 0} \log \left|\frac{(1-q_{z_0} q_{-\tau}^{n})^N}{1-q_{Nz_0}q^{n}_{-N\tau}} \right|
\end{split}
\end{equation*}
and so their sum equals
$
-\log |\theta_0^{(N)}(z_0,-\tau)|
$
where we write $ \theta_0^{(N)}(z,-\tau) = \theta_0(z,-\tau)^N/\theta_0(Nz,-N \tau)$. Choosing similarly $\beta \in H(a)$ such that $\beta(b)>0$ and $L \cap H(a)=\Z \beta + \Z \gamma$ we find that the contribution to the boundary term from vectors lying on $H(a)$ is $\log|\theta_0^{(N)}(z_0,\sigma)|$ where $\sigma:=\beta(\mathbf{x})/\gamma(\mathbf{x})$ has positive imaginary part since $(\beta,\gamma)$ is an oriented basis of $H(a)$. We conclude that
$$
\text{boundary term in } \eqref{eq:log_gamma_plus_bdy_term} = \log \left|\frac{\theta_0^{(N)}(z_0,\sigma)}{\theta_0^{(N)}(z_0,-\tau)} \right|.
$$

\begin{lemma} \label{lemma:bnd_term}
We have
$\log \left|\theta_0^{(N)}(z_0,\sigma)/\theta_0^{(N)}(z_0,-\tau) \right|=0$.
\end{lemma}
\begin{proof}
Since $b=\varepsilon a$, multiplication by $\varepsilon$ induces an isomorphism $L \cap H(a) \simeq L \cap H(b)$. Hence we can write
\begin{equation*}
\begin{split}
\varepsilon \beta & = -a' \alpha + b' \gamma \\
\varepsilon \gamma & = -c \alpha + d \gamma
\end{split}
\end{equation*}
for some integers $a',b',c, d$ with $a'd-b'c=\pm 1$. Evaluating at $\mathbf{x}$ and writing $\varepsilon_\C=\varepsilon(\mathbf{x}) \in \mathbf{C}^\times$, we find
\begin{equation*}
\begin{split}
\varepsilon_\C \beta(\mathbf{x}) & = (\varepsilon \beta)(\mathbf{x}) = -a' \alpha(\mathbf{x}) + b' \gamma(\mathbf{x}) \\
\varepsilon_\C \gamma(\mathbf{x}) & = (\varepsilon \gamma)(\mathbf{x}) = -c \alpha(\mathbf{x}) + d \gamma(\mathbf{x}).
\end{split}
\end{equation*}
Dividing by $\gamma(\mathbf{x})$ this gives
\begin{equation*}
\begin{split}
\sigma & = \frac{a' \cdot (-\tau) + b'}{c \cdot(-\tau) + d} \\
\varepsilon_\C &= c \cdot (-\tau) + d .
\end{split}
\end{equation*}
Since $\sigma$ and $-\tau$ have positive imaginary part, this implies that $a'd-b'c=1$. Note that $\mathfrak{a}^{-1}L \cap H(a) = \mathbf{Z} \beta + \mathbf{Z}(\gamma/N)$ and $\mathfrak{a}^{-1}L \cap H(b) = \mathbf{Z} \alpha + \mathbf{Z}(\gamma/N)$. Since multiplication by $\varepsilon$ preserves $\mathfrak{a}^{-1}L$, it follows that $N$ divides $c$. We also have $\lambda/q \equiv 1 \mod L$ and $\varepsilon  \equiv 1 \mod \mathfrak{f}$ and hence
\begin{equation*}
\begin{split}
\varepsilon \frac{\lambda}{q} = \varepsilon + \varepsilon \left( \frac{\lambda}{q}-1 \right) \in \varepsilon + L = \frac{\lambda}{q}+L.
\end{split}
\end{equation*}
Thus $n(\varepsilon \gamma - \gamma)/q = (\epsilon \lambda-\lambda)/q \in L \cap H(b)$, giving $nc \equiv 0 \mod q$ and $n(d-1) \equiv 0 \mod q$. 

The statement now follows from the uniqueness principle in \cite[Prop. 1.3]{Kato}. For $z \in \mathbf{C}$, the function $\theta_0(\cdot, -\tau)$ satisfies $\theta_0(z+1,-\tau)=\theta_0(z,-\tau)$ and  $\theta_0(z-\tau,-\tau)=-e^{-2\pi i z}\theta_0(z,-\tau)$, and also
$$
\theta_0(Nz,-N\tau) = \prod_{j=0}^{N-1} \theta_0(z+j/N,-\tau).
$$
It follows from these properties that the function $\theta_0^{(N)}(\cdot, -\tau)$ defined by
$$
\theta_0^{(N)}(z,-\tau) = \theta_0(z,-\tau)^N/\theta_0(Nz,-N \tau)
$$
is periodic with respect to the lattice $\Z+\Z\tau$ (note $N$ is odd) and so it defines a meromorphic function on the elliptic curve $E_\tau = \C/\Z+\Z \tau$. The divisor of this function is $N \cdot [0]-\sum_{j=0}^{N-1}[j/N]$ (since  $\theta_0(\cdot, -\tau)$ is entire with simple zeroes at $\Z+\Z \tau$).

Consider now the function
$$
\tilde{\theta}_0^{(N)}(z,-\tau) := \theta_0^{(N)}\left(\frac{z}{c\cdot (-\tau)+d},\frac{a' \cdot (-\tau)+b'}{c \cdot(-\tau)+d} \right).
$$
As with $\theta_0$ one shows that $\tilde{\theta}_0^{(N)}(\cdot,-\tau)$ is periodic for the lattice $\langle a' \cdot (-\tau)+b', c \cdot (-\tau)+d \rangle = \Z + \Z \tau$, and using that $N|c$ one sees that the resulting function on $E_\tau$ also has divisor $N \cdot [0]-\sum_{j=0}^{N-1}[j/N]$. It follows that
$$
\tilde{\theta}_0^{(N)}(z,-\tau) = C \cdot \theta_0^{(N)}(z,-\tau)
$$
for some constant $C \in \C^\times$. We claim that $C^3 =1$ for any odd $N$ and that $C=1$ if $(6,N)=1$. This follows from the following relations that one can check directly using \eqref{eq:def_theta_0}: if $(l,N)=1$, then
\begin{equation*}
\begin{split}
\prod_{m,n \in \Z/l\Z} \theta_0^{(N)}\left(\frac{z}{l}+\frac{m+n\cdot(-\tau)}{l}, -\tau  \right) &= \theta_0^{(N)}(z,-\tau) \\
\prod_{m,n \in \Z/l\Z} \tilde{\theta}_0^{(N)}\left(\frac{z}{l}+\frac{m+n\cdot(-\tau)}{l}, -\tau  \right) &= \tilde{\theta}_0^{(N)}(z,-\tau).
\end{split}
\end{equation*}
These imply that $C^{l^2}=C$ for any positive integer $l$ coprime to $N$. Thus for odd $N$ we have $C^4=C$, and if $(6,N)=1$, then the equations $C^4=C^9=C$ imply $C=1$.

We conclude that
$$
\theta_0^{(N)}(z,\sigma) = \zeta \cdot \theta_0^{(N)}(z(-c \tau+d),-\tau)
$$
with $\zeta^3=1$ (and $\zeta = 1$ when $(6,N)=1$). For $z_0=n/q$, the congruence properties of $c$ and $d$ modulo $q$ show that $z_0(-c \tau +d) \equiv z_0 \mod \Z + \Z \tau$, finishing the proof.
\end{proof}
Combining the above Lemmas we conclude that
\begin{equation}
\begin{split}
\zeta'_{\mathfrak{f},\mathfrak{a}}(\mathfrak{b},0) &= \frac{3}{\sqrt{\pi}} \int_0^\infty \int_0^\infty F_{\mathfrak{f},\mathfrak{a}}(t_1,t,h,\mathfrak{b}) \frac{dt_1}{t_1} \frac{dt}{t} \\
&= \frac{3}{\sqrt{\pi}} (I_{\texttt{+-}}(h)+I_{\texttt{-+}}(h)) \\
& =  \log|\mathbf{\Gamma}_{\mathfrak{a},h}(\mathfrak{fb}^{-1})|^2,
\end{split}
\end{equation}
proving Theorem \ref{T:KLintro}.

\medskip
\noindent
{\it Remark.}
The proof of Lemma \ref{lemma:bnd_term} implies the following property of the numbers $\mathbf{\Gamma}_{\mathfrak{a},h}(L)$. Note that the function $\Gamma_{a,b}$ defined in \eqref{felder-v-def} has the property
\begin{equation*}
\begin{split}
& \Gamma_{a,b}(-w,z;L) \Gamma_{a,b}(w,z;L) \\ &= \frac{\prod_{\delta \in (C_{\texttt{+-}}(a,b)\cap H(b))/\Z \gamma} \left(1-e^{-2\pi i (\delta (z) -w ) / \gamma (z)} \right) \left(1-e^{-2\pi i (\delta (z) +w ) / \gamma (z)} \right) }{\prod_{\delta \in (C_{\texttt{-+}}(a,b)\cap H(a))/\Z \gamma} \left(1-e^{-2\pi i (\delta (z) -w ) / \gamma (z)} \right) \left(1-e^{-2\pi i (\delta (z) +w ) / \gamma (z)} \right) }.
\end{split}
\end{equation*}
Using the notation in the above proof, this implies that
\begin{equation}  \label{remark:evaluation_-h}
\mathbf{\Gamma}_{\mathfrak{a},h}(L)\mathbf{\Gamma}_{\mathfrak{a},h}(-h(\mathbf{x}),\mathbf{x};L) = \frac{\theta_0^{(N)}(z_0,\sigma)}{\theta_0^{(N)}(z_0,-\tau)} = \zeta,
\end{equation}
where $\zeta^3=1$; moreover, $\zeta=1$ if $(6,\mathrm{N}(\mathfrak{a}))=1$.

\subsubsection{A geometric interpretation of the elliptic gamma function}

Felder--Henriques--Rossi--Zhu \cite{Felder} gave a geometric interpretation of the elliptic gamma function. A \emph{triptic curve} is a pointed holomorphic stack of the the form $[\C / \Z^3 ]$ where the action is induced by a homomorphism $\Z^3 \to \C$ of real rank $2$. It plays a role analogous to that of the elliptic curve for the theta function. The moduli space of oriented triptic curves is the complex stack 
\begin{equation} \label{Stack}
\left[ \SL_3 (\Z ) \backslash (\C \mathrm{P}^2 - \R \mathrm{P}^2) \right].
\end{equation}
The elliptic gamma function is a section of a hermitian abelian gerbe over the universal oriented triptic curve. Our construction can be understood in terms of a modified version of this gerbe, defined using the ideal $\mathfrak{a}$. The stack \eqref{Stack} is however far from being algebraic, thus proving our conjecture certainly requires new ideas. We nevertheless want to point out that --- up to replacing $\SL_3 (\Z)$ by $\SL_2 (\Z[1/p])$ and  $\C \mathrm{P}^2 - \R \mathrm{P}^2$ by the product of $\mathcal{H}$ by the Drinfeld upper half-plane --- the above geometric picture seems quite parallel to that of the recent work of Darmon and Vonk \cite{DarmonVonk} on explicit class field theory for real quadratic fields.

\section{A few additional numerical examples}\label{section:addition_numerical_examples}

In this Section we present numerical evidence for Conjecture \ref{S:conj} in some examples. The computation relies on the summation formula \cite[eq. (15)]{FelderVarchenko} for the logarithm of the elliptic gamma function, whose convergence is slow if $\tau$ and $\sigma$ have small imaginary parts. This is the main bottleneck towards computing with larger input data. It would be interesting to find faster methods to evaluate $\mathbf{\Gamma}_{\mathfrak{a},h}(\mathfrak{fb}^{-1})$, for example by making more efficient use of the $\mathrm{SL}_3(\Z)$-modularity of $\Gamma(z,\tau,\sigma)$ \cite{FelderVarchenko}.

\subsection{$K= \Q (\beta)$ where $\beta^3-3 = 0$} 

The discriminant is $-3^5$, the class number is $1$ and $\beta^2 -2$ is a fundamental unit. Take $\mathfrak{f}$ to be the unique ideal of norm $3$. Then the fundamental unit belongs to $\mathcal{O}_{\mathfrak{f}}^{+,\times}$.

The narrow ray class field $K(\mathfrak{f})$ is a quadratic extension of $K$. We begin with this simple example because we carried out our first tests on it. 

We first take the smoothing ideal $\mathfrak{a}$ to be the single prime ideal with norm $5$. Then the complex numbers $\mathbf{\Gamma}_{\mathfrak{a} , h} (  \mathfrak{f} )$ appear to be independent of $h$ and equal to $ 0.18108\ldots + i \cdot 0.21746\ldots$, a complex number that Pari-GP identifies as a root of 
$$x^6 - 3 x^5 +6x^4 +17x^3 +6x^2 - 3 x+1$$
to 1000 digits. Here again the splitting field is the quadratic extension $K(\mathfrak{f})$ of $K$. Since the class number formula implies Stark's conjecture in this setting \cite[Thm. 5.4]{Tate}, our Theorem \ref{T:KLintro} shows that $|\mathbf{\Gamma}_{\mathfrak{a},h}(\mathfrak{f})|$ is algebraic.

Let us now take the smoothing ideal $\mathfrak{a}$ to be the single prime ideal with norm $2$. The \texttt{algdep} command of Pari-GP allows to show that the complex numbers $\mathbf{\Gamma}_{\mathfrak{a} , h} (\mathfrak{f} )$ are very close to a root of a degree $24$ irreducible polynomial over $\Q$. As predicted by one of the remarks following our conjecture, up to multiplication by $\pm 1$, the complex number $\mathbf{\Gamma}_{\mathfrak{a} , h} ( \mathfrak{f} )^4$ is independent of $h$, and Pari-GP identifies it as $\pm \mathbf{u}_{2}$ where $\mathbf{u}_2$ is a root of 
$$x^6 + 24 x^5 + 168 x^4 + 98 x^3 + 168 x^2 + 24 x + 1$$
to 1000 digits. The splitting field of this polynomial is the quadratic extension $K(\mathfrak{f})$ of $K$.

\subsection{$K= \Q (\beta)$ where $\beta^3-7 = 0$} 

The discriminant is $-3^3 \cdot 7^2$, the class number is $3$ and $2-\beta$ is a fundamental unit. Take $\mathfrak{f}$ to be the unique ideal of norm $3$. The corresponding narrow ray class field $K(\mathfrak{f})$ has relative degree $6$ over $K$ and the fundamental unit is a generator of $\mathcal{O}_{\mathfrak{f}}^{+,\times}$.

In this case it turns out that one can find many choices of $h$ and $L$ such that that the set $F/\Z \gamma$ appearing in \eqref{eq:Gamma_a_b_as_prod_over_F} is a singleton. This explains why there is only one occurence of the basic elliptic gamma function $\Gamma$ in the example of the introduction. There the smoothing ideal $\mathfrak{a}$ is the single ideal with norm $5$, and the complex numbers $\mathbf{\Gamma}_{\mathfrak{a} , h} ( \mathfrak{f} )$ appear numerically to be independent of $h$.

Let $\mathfrak{b}$ be the unique ideal of norm $2$. Its class is a generator of the ray class group. We identify the alleged units $\mathbf{u}_{ \mathfrak{f}\mathfrak{b}^{-k}, \mathfrak{a}}, \ (0 \leq k \leq  5)$, with their complex images given by the elliptic gamma values.
As predicted by our conjecture one verifies that these numbers coincide with the roots of \eqref{Intro:pol} to 1000 digits and that these roots satisfy parts (2) and (3) of our conjecture. The number \eqref{Intro:sol} is $\mathbf{u}_{ \mathfrak{fb}^{-5} , \mathfrak{a}}$ (taking $\beta \in \C$ to have negative imaginary part). It is very close to the image, under $\sigma_{\mathfrak{b}}^{-1}$, of the root of \eqref{Intro:pol} that is 
$$\mathbf{u}_{\mathfrak{f} , \mathfrak{a}} \approx
-0.06261629382 \ldots - i \cdot 0.2669850711 \ldots .
$$

We have similarly tested our conjecture with other smoothing ideals, e.g.  $\mathfrak{a}_{11}$  the single prime ideal with norm $11$.

\subsection{$K= \Q (\beta)$ where $\beta^3-2 = 0$} 

The discriminant is $-108$, the class number is $1$ and $\beta -1$ is a fundamental unit. Fix a complex embedding $K \hookrightarrow \C$.

Following Ren and Sczech \cite[\S 4.1]{RenSczech} we first take $\mathfrak{f} = (3)$. The corresponding narrow ray class field $K(\mathfrak{f})$ has relative degree $6$ over $K$ and $(\beta-1)^3$ is a generator of $\mathcal{O}_{\mathfrak{f}}^{+, \times}$. The field $K(\mathfrak{f})$ contains $18$ roots of unity. We fix an embedding $K(\mathfrak{f}) \hookrightarrow \C$ extending the fixed complex embedding of $K$.

Let our smoothing ideal $\mathfrak{a}$ be the single prime ideal with norm $5$. The class of the ideal $\mathfrak{b} = (\beta)$ generates the ray class group. Using our conjecture we obtain complex numbers that coincide with the roots of 
\begin{equation*}
\begin{split}
x^6 &-(3\beta^2-9\beta+6)x^5 +(23\beta^2-17\beta-16)x^4 \\ & + (48\beta^2 -43 \beta -23)x^3 +(23\beta^2-17\beta-16)x^2 - (3\beta^2-9\beta +6)x +1
\end{split}
\end{equation*}
to $1000$ decimal digits.   This polynomial is irreducible over $K$ with splitting field  $K(\mathfrak{f}).$ Its image under the real embedding of $K$ has all its roots lying on the unit circle.
 
Ren and Sczech \cite{RenSczech} have checked numerically that among the eighteen roots of 
\begin{multline*}
x^{18} - 7767x^{17} + 51550065x^{16} - 199524692622x^{15} + 520755985257966x^{14} \\ - 1828056747902004x^{13} + 24870880029533226x^{12} - 80588629212013080x^{11} \\ + 116076408275027511x^{10} - 118102314911180623x^9 + 116076408275027511x^8 \\ - 80588629212013080x^7 + 24870880029533226x^6 - 1828056747902004x^5  \\ + 520755985257966x^4 - 199524692622x^3 + 51550065x^2 - 7767x + 1,
\end{multline*}
six have absolute value $1$ and the twelve others are the complex images of the Galois transforms of ``the'' alleged Stark unit $\mathbf{u}_{\rm Stark}$ (uniquely defined up to the $18$ roots of unity in $K(\mathfrak{f})$) that is denoted by $\eta_6$ in \cite{RenSczech}. We verify numerically that
$$
\mathbf{u}_{\rm Stark}^{\mathrm{N} (\mathfrak{a})-\sigma_{\mathfrak{a}}} \approx \mathbf{\Gamma}_{\mathfrak{a},h}(\mathfrak{f})^{18}.
$$

\medskip

We have also numerically checked  other ramification ideals, e.g. $\mathfrak{f} = (5)$.

\subsection{$K= \Q (\beta)$ where $\beta^3 -\beta+1=0$} It is the complex cubic field with minimal discriminant, equal to $-23$. It is of 
class number $1$ and $-\beta$ is a fundamental unit.  Fix an embedding $K \hookrightarrow \C$.

As in \cite[\S 4.2]{RenSczech} we consider $\mathfrak{f} = (5)$. The corresponding narrow ray class field $K(\mathfrak{f})$ has relative degree $4$ over $K$ and $\beta^{24}$ is a generator of $\mathcal{O}_{\mathfrak{f}}^{+, \times}$. The field $K(\mathfrak{f})$ contains $10$ roots of unity. We fix an embedding $K(\mathfrak{f}) \hookrightarrow \C$ extending the fixed complex embedding of $K$.

Let our smoothing ideal $\mathfrak{a}$ be the single prime ideal with norm $7$. Using our conjecture we  obtain complex numbers that coincide with the complex roots of 
$$x^4+(70 \beta^2+15 \beta -104)x^3-(65 \beta^2+155 \beta + 89)x^2+(70 \beta^2+15 \beta -104)x+1$$
to $1000$ decimal digits.   This polynomial is irreducible over $K$ with splitting field  $K(\mathfrak{f}).$ Its image under the real embedding of $K$ has all its roots lying on the unit circle.

Ren and Sczech have checked numerically that, in this example, ``the'' alleged Stark unit $\mathbf{u}_{\rm Stark}$ (uniquely defined up to the $10$ roots of unity in $K(\mathfrak{f})$) is a root of 
$$x^4 - (395\beta^2 - 60 \beta - 776) x^3 + (495 \beta^2 - 385 \beta -1374) x^2 -  (395\beta^2 - 60 \beta - 776) x +1.$$
It is denoted by $\eta_4$ in \cite{RenSczech}. 
We verify numerically that
$$
\mathbf{u}_{\rm Stark}^{\mathrm{N} (\mathfrak{a})-\sigma_{\mathfrak{a}}} \approx \mathbf{\Gamma}_{\mathfrak{a},h}(\mathfrak{f})^{10}.
$$

\section*{Appendix: Eisenstein's Jugendtraum}

This section provides some historical background underpinning our work, 
deeply rooted in Gotthold Eisenstein's largely ignored paper \cite{Eisenstein}.

In 1844, three years before his masterpiece \cite{eis2} laying alternative foundations for a theory of elliptic functions, Eisenstein published in Crelle's journal a short prospective note \cite{Eisenstein} where he first  sketches his own  construction  of classical trigonometric and elliptic functions by considering nested conditionally convergent infinite products.

In the last pages, he indicates how one could go beyond and investigate  higher dimensional cases.
Aware of the fact that no meromorphic functions over $\mathbf C$ can be triply periodic, he first states: 

\begin{quote}
 Following the remark [of his \S 1] related to simple and  double infinite products relative to circular and elliptic functions, we should consider as an analog of greater degree the \textit{quotients of quotients of infinite triple products}\footnote{Eisenstein's  emphasis; this sentence  has been  ridiculed by Jacobi in \cite{Jacobi}, who misjudged  Eisenstein's  original point of view; see also \cite{Adler}.} of the form
\begin{equation}\label{eisprod}\prod \left(1-\frac{x}{\lambda+\lambda'A+\lambda'' A'}\right),\end{equation} where $A$ and $A'$ are constants and $\lambda, \lambda', \lambda''$ represent indices with the ordinary meaning attached to them.  [...] 
 
Unlike the simple or double products, these multiple products would not have a determined analytical sense if one wishes to
assign to the indices all possible values.'' 
\end{quote}
Eisenstein goes further and suggests:
\begin{quote}
Nevertheless nothing  refrains from considering  these multiple products  as soon as the indices are constrained by certain \textit{restrictions} that enable to discard precisely  those circumstances  contrary to  convergence  \nolinebreak  : for example, certain \textit{inequality conditions} could be assigned to  the  indices $\lambda, \lambda', \lambda''$
and adjoined to the product, to the effect that the multiplication sign is extended to the values of the indices satisfying these conditions, the other values being omitted. Generally speaking, one cannot say here anything further about  the choice of these conditional equations. There exists however a whole class of similar functions which have very close connections to certain results in Number Theory.
\end{quote}

\medskip 

According to  this paragraph, it looks plausible  to us that Eisenstein  had  anticipated the construction of (quotients of) the elliptic gamma function by means of the  simple manipulations for multiple infinite products he pioneered. As we shall now exemplify, 
Eisenstein's own convergence method  quickly allows for an explicit  relation  between  a triple  product of the shape (\ref{eisprod}), limited by the \textit{inequality conditions}
$\{(\lambda, \lambda',\lambda'')\in \mathbf Z^3 \; \big| \;  \lambda',\lambda'' >0\},$ and  the function $\Gamma(z, A, A').$  \medskip

Recall (see for instance \cite{Weil,CS} for a survey) that in dimension $1$ his method provides  the classical identity for trigonometric functions 
\begin{equation*}
\begin{split}
\prod_{\lambda\in \mathbf Z}^{(\textrm{eis})}\left(1-\frac{z}{\lambda+u}\right)& =  \lim_{M\rightarrow +\infty}
\prod_{\lambda=-M}^M \left(1-\frac{z}{\lambda+u}\right)
\\
& = 
\frac{\sin \pi(u-z)}{\sin(\pi u)}\\ 
& = e^{i\pi z}\left(\frac{1-e^{2\pi i (u-z)}}{1-e^{2\pi i u}}\right).
\end{split}
\end{equation*}
In order to improve the convergence of the next infinite products,  it is convenient to get rid of the exponential factor by introducing a smoothed version with an extra integral parameter $N>1 :$ 

\begin{equation*}
\prod_{\lambda\in \mathbf Z}^{(\textrm{eis})}\frac{\left(1-\frac{z}{\lambda+u}\right)^N }{\left(1-\frac{Nz}{\lambda+Nu}\right)}  = \frac{(1-e^{2\pi i (u-z)})^N (1-e^{2\pi i Nu})}{(1-e^{2\pi i u})^N(1-e^{2\pi i (Nu-Nz)})}\cdot 
\end{equation*}

Now we substitute $u=A \lambda'+A'\lambda''+ v,$ with $A, A'\in \mathcal H, $  and proceed to the absolutely convergent double product over $\lambda',\lambda''\in \mathbf Z$  constrained by the inequality $\lambda',\lambda''>0. $ This  immediately leads  to 

\begin{align*} \label{eistoGamma}
Q_{\texttt{++}}=&\, \prod_{\lambda\in \mathbf Z, \lambda',\lambda''>0}^{(\textrm{eis})} 
\frac{\left(1-\frac{z}{\lambda+A \lambda'+A'\lambda''+ v}\right) ^N }{\left(1-\frac{Nz}{\lambda+NA \lambda'+NA'\lambda''+ Nv}\right)} \\ 
 =&\,  \prod_{\lambda',\lambda''>0} \frac{(1-e^{2\pi i (A \lambda'+A'\lambda''+ v-z)})^N 
 (1-e^{2\pi i (NA \lambda'+NA'\lambda''+N v) }) }{(1-e^{2\pi i (A \lambda'+A'\lambda''+ v)})^N
(1-e^{2\pi i (NA \lambda'+NA'\lambda''+N v-Nz) } )}\cdot
\end{align*}
We write $Q_{\texttt{--}}$ for the similar ratio, now indexed by $\lambda',\lambda''\in \Z$ satisfying the inequalities 
$\lambda',\lambda''\leq 0.$ It is then straightforward to express the quotient of quotients $Q_{\texttt{++}}/Q_{\texttt{--}}$  solely in terms of elliptic gamma functions as 

\begin{equation}\label{oneEisensteinquotientofquotients}
\frac{Q_{\texttt{++}}}{Q_{\texttt{--}}}=\frac{\Gamma(z-v,A,A')^N/\Gamma(-v,A,A')^N}{\Gamma(Nz-Nv,NA,NA')/\Gamma(-Nv,NA,NA')}\cdot
\end{equation}

\medskip

Additionally, Eisenstein highlights  as a case of special interest for Number Theory when the inequalities  correspond to  grouping together values of the indices ``in geometric progression.''
We are inclined following \cite{RenSczech} to infer that Eisenstein could  have foreseen applications to the arithmetic of cubic complex fields, for which the unit group is of rank one.
 
He concludes:

\begin{quote}
Pursuing  this path, the functions one is led to seem to possess  very remarkable properties; they open up a field in which   the most prolific investigations  offer themselves and appear to be ground where  the most difficult parts 
of Analysis and Number Theory join.
\end{quote}

This cryptic prediction,``as good as lost"  as many original ideas of Eisenstein (\cite{Weil}, p.4), predates  the utterance of Kronecker's Jugendtraum by 36 years.  As the very name indicates, Kronecker must have carried his own aspiration over a long period of time. We hope our paper provides some evidence that their two programs, maybe conceived while they were together students in Berlin,  have tight connections. 
It is however  to be observed  that only Eisenstein's text  points to specific functions (``infinite triple products'') to pursue his dream beyond the $\GL_2$ setting. 

Considering that Eisenstein was only 21 when he outlined this rough sketch of a  program, and that the arithmetic applications we have unveiled are, in the case of complex cubic fields, quite close to Kronecker's Jugendtraum, or Hilbert's twelfth problem, 
we felt entitled to name it ``Eisenstein's Jugendtraum.''

\smallskip

As a closing note, let us recast our initial numerical data somewhat  differently
to confirm that Eisenstein's meromorphic quotients of quotients 
(\ref{oneEisensteinquotientofquotients}) undoubtedly  possess striking number theoretic properties.
In parallel to (\ref{Intro:sol}), it is legitimate to  choose the specializations $z=1/3$, $v=-1/3$, $N=5$ together with 
$A=-(2\beta+\beta^2)/15$, $A'=-(2+\beta)/15$ where $\beta=\sqrt[3]{7}e^{-\frac{2i\pi}{3}}.$
We deduce that, even if neither $Q_{\texttt{++}}$ nor $Q_{\texttt{--}}$ seem to be algebraic individually, their ratio 

\begin{equation}
\frac {Q_{\texttt{++}}}{Q_{\texttt{--}}}\approx -1735.727884\ldots+i\cdot 336.859173\ldots
\end{equation}
does. To $1000$ digits of accuracy, this is the square of the  root (\ref{Intro:sol}) of the degree six polynomial with coefficients in $\Z[\beta]$ displayed in the Introduction.

\bibliographystyle {plain}
\bibliography{BCG_cubic}
\end{document}